\documentclass[10pt,a4paper]{article}
\usepackage{latexsym}
\usepackage{amssymb}
\usepackage{amsmath}
\usepackage{amsfonts}
\usepackage{amsthm}
\usepackage{graphicx}
\usepackage{psfrag}

\def\R{\mathbb{R}}
\def\C{\mathbb{C}}
\def\N{\mathbb{N}}

\def\Z{\mathbb{Z}}
\def\t{\mathrm{t}}

\def\c{\mathrm{c}}

\def\P{\mathbb{P}}
\def\RR{\mathcal{R}}

\theoremstyle{plain}
\newtheorem{thm}{\noindent Theorem}
\newtheorem{cor}[thm]{\noindent Corollary}  
\newtheorem{lem}[thm]{\noindent Lemma}
\newtheorem{prop}[thm]{\noindent Proposition}

\theoremstyle{definition}
\newtheorem{defn}{\noindent Definition}
\newtheorem{example}{\noindent Example}

\title{Computing the Newton Polygon \\ of the Implicit Equation}

\author{Ioannis Z.~Emiris\thanks{Dept.\ of Informatics and Telecommunications,
        National Kapodistrian University of Athens, Greece.}
        \and Christos Konaxis\footnotemark[\value{footnote}]
        \and Leonidas Palios\thanks{Dept.\ of Computer Science, University of Ioannina, Greece}
}

\date{17 November 2007}

\begin{document}
\maketitle

\begin{abstract}
We consider polynomially and rationally parameterized curves, where
the polynomials in the parameterization have fixed supports and generic coefficients.
We apply sparse (or toric) elimination theory in order to determine the
vertex representation of its implicit polygon, i.e.\
of the implicit equation's Newton polygon.
In particular, we consider mixed subdivisions of the input Newton polygons and
regular triangulations of point sets defined by Cayley's trick.
We distinguish polynomial and rational parameterizations, where the latter may have
the same or different denominators; the implicit polygon is shown to have,
respectively, up to~4, 5, or~6 vertices.
In certain cases, we also determine some of the coefficients in the implicit equation.
\end{abstract}


\section{Introduction}

Implicitization is the problem of switching from a parametric representation of a
hypersurface to an algebraic one.
It is a fundamental question with several applications, e.g.~\cite{Hof89,HoSeWi97}.
Here we consider the implicitization problem for a planar curve, where the
polynomials in its parameterization have fixed Newton polytopes.
We determine the vertices of the Newton polygon of
the implicit equation, or \emph{implicit polygon}, without computing the equation,
under the assumption of \emph{generic} coefficients relative to the given supports,
i.e.\ our results hold for all coefficient vectors in some open dense subset of
the coefficient space.

This problem was posed in \cite{SY94}. It appeared in~\cite{EmiKot03,EmiKot05},
then in~\cite{STY,SY07}, and more
recently in~\cite{Esterov-Khovanskii} and~\cite{DAndrea-Sombra}.
The motivation is that
\lq\lq apriori knowledge of the Newton polytope
would greatly facilitate the subsequent computation of recovering the coefficients
of the implicit equation [\ldots] This is a problem of numerical linear
algebra \dots\rq\rq [{}\cite{STY}].
Reducing implicitization to linear algebra is also the premise of \cite{CoGiKoWa,EmiKot03}.
Yet, this can be nontrivial if coefficients are not generic.
Another potential application of knowing the implicit polygon is to approximate
implicitization, see \cite{Dokken}.

Previous work includes \cite{EmiKot03,EmiKot05}, where
an algorithm constructs the Newton polytope of any implicit equation.
That method had to compute all mixed subdivisions,
then applies cor.~\ref{C:surjection-mxd-extreme}.
In \cite[chapter~12]{GKZ}, the authors study the resultant of two univariate polynomials
and describe the facets of its Newton polytope.
In~\cite{GKZ90}, the extreme monomials of the Sylvester resultant are described.
%
The approaches in \cite{EmiKot03,GKZ},
cannot exploit the fact that the denominators in a rational parameterization may be identical.

\cite{STY} offered algorithms to compute the Newton polytope
of the implicit equation of any hypersurface parameterized by Laurent polynomials.
Their approach is based on tropical geometry.
It extends to arbitrary implicit ideals.
They give a generically optimal implicit support;
for curves, the support is described in~\cite[example~1.1]{STY}.
Their approach also handles rational parameterizations with the same denominator by
homogenizing the parameter as well as the implicit space.
The implicit equation is homogeneous, hence its Newton polytope lies in a hyperplane,
which may cause numerical instability in the computation.

In~\cite{Esterov-Khovanskii} the problem is solved in an abstract way by means of
composite bodies and mixed fiber polytopes.
In~\cite{DAndrea-Sombra} the normal fan of the implicit polygon is determined,
with no genericity assumption on the coefficients.  This is computed by the
multiplicities of any parameterization of the rational plane curve.
The authors are based on a refinement of the famous Kushnirenko-Bernstein formula
for the computation of the isolated roots of a polynomial system in the torus,
given in~\cite{Philippon-Sombra07}.
As a corollary, they obtain the optimal implicit polygon in case of generic coefficients.
They also address the inverse question, namely when can a given polygon be the
Newton polygon of an implicit curve.

In many applications, such as computing the $u$-resultant or in implicitization,
the resultant coefficients are themselves polynomials in a few parameters,
and we wish to study the resultant as a polynomial in these parameters.
In~\cite{EmKoPa07}, we computed the Newton polytope of specialized resultants
while avoiding to compute the entire secondary polytope;
our approach was to examine the silhouette of the latter with respect
to an orthogonal projection.
We presented a method to compute the vertices of the implicit polygon
of polynomial or rational parametric curves, when denominators differ.
We also introduced a method and gave partial results for the case when denominators
are equal; the latter method is fully developed in the present article.

Our main contribution is to determine the vertex structure of the implicit polygon.
This polygon is optimal if the coefficients of the parametric polynomials are sufficiently
generic with respect to the given supports, otherwise it contains the true polygon.
Our presentation is self-contained.
In the case of polynomially parameterized curves and rationally parameterized curves with
different denominators (which includes the case of Laurent polynomial parameterizations),
the Cayley trick reduces the problem to
computing regular triangulations of point sets in the plane.
In retrospect, our methods are similar to those employed in~\cite{GKZ90}.
We also determine certain coefficients in the implicit equation.
If the denominators are identical, two-dimensional mixed subdivisions are examined;
we show that only subdivisions obtained by {\em linear} liftings are relevant.

The following proposition collects our main corollaries regarding
the shape of the implicit polygon in terms of corner cuts on an initial polygon:
$\phi$ is the implicit equation and $N(\phi)$ is the implicit polytope.
\begin{prop}\label{P:Nf-cuts}
$N(\phi)$ is defined by a polygon with one vertex at the origin
and two edges lying on the axes. In particular,
\\ Polynomial parameterizations:
$N(\phi)$ is defined by a right triangle 
with at most one corner cut, which excludes the origin.
\\ Rational parameterizations with equal denominators:
$N(\phi)$ is defined by a right triangle 
with at most two cuts, on the same or different corners.
\\ Rational parameterizations with different denominators:
$N(\phi)$ is defined by a quadrilateral 
with at most two cuts, on the same or different corners.
\end{prop}

\begin{example}\label{Exam:intro}
Consider:
$$
x=\frac{t^6+2t^2}{t^7+1} , y=\frac{t^4-t^3}{t^7+1} ,
$$
Theorem~\ref{Tcase2a} yields vertices
$(7,0),(0,7),(0,3),(3,1),(6,0)$, which define the actual implicit polygon
because the implicit equation is\\[-0.5em]

$
\phi= -32y^4-30x^3y^2-x^4y-12x^2y^2-3x^3y-7x^6y-2x^7+20xy^3+280x^2y^5-7^3y^4x-70x^4y^3-\\
22x^3y^3-49x^5y^2-21x^4y^2+11x^5y+216y^5+129y^7-248y^6+70xy^6+185xy^5+24y^3+100xy^4+\\
43x^2y^3+72x^2y^4+3x^6.\\[-0.5em]
$

Changing the coefficient of $t^2$ to -1, leads to an implicit polygon
with 4 cuts which is contained in the polygon predicted by theorem~\ref{Tcase2a}.
This shows the importance of the genericity condition on the coefficients of the
parametric polynomials. See example~\ref{Exam:samedenom_5vert} for details.

An instance where the implicit polygon has~6 vertices is:
$$
x=\frac{ t^3+2t^2+t }{ t^2+3t-2 }, ~ y= \frac{ t^3-t^2 }{ t-2 }.
$$
Our results in section~\ref{S:Rat_dif_denom} yield implicit vertices
$\{(0,1), (0,3), (3,0), (1,3), (2,0), (3,2)\}$
which define the actual implicit polygon.
See example~\ref{EX:DAndrea} for details.

\begin{figure}[ht]
\centering\includegraphics[width=0.35\textwidth]{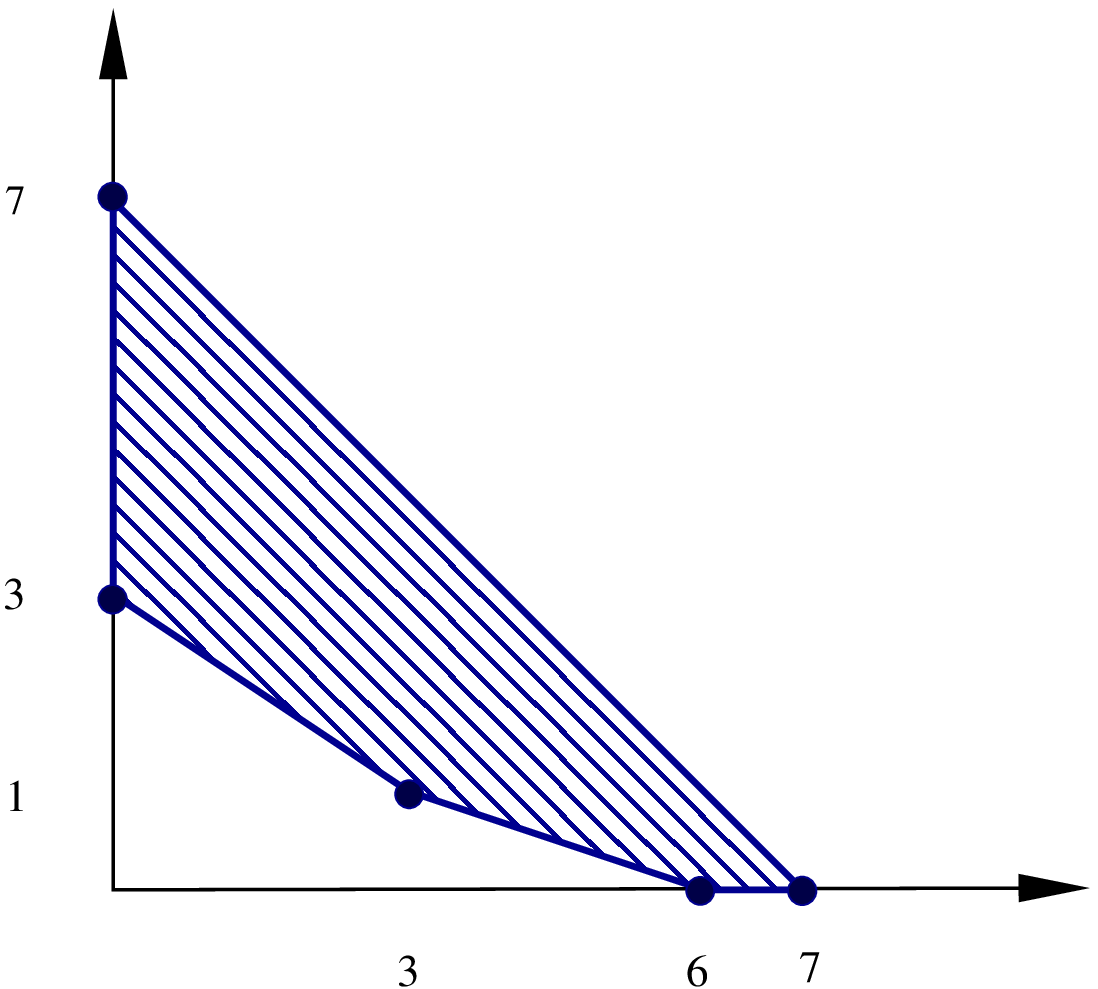}
\hskip4em
\includegraphics[width=0.32\textwidth]{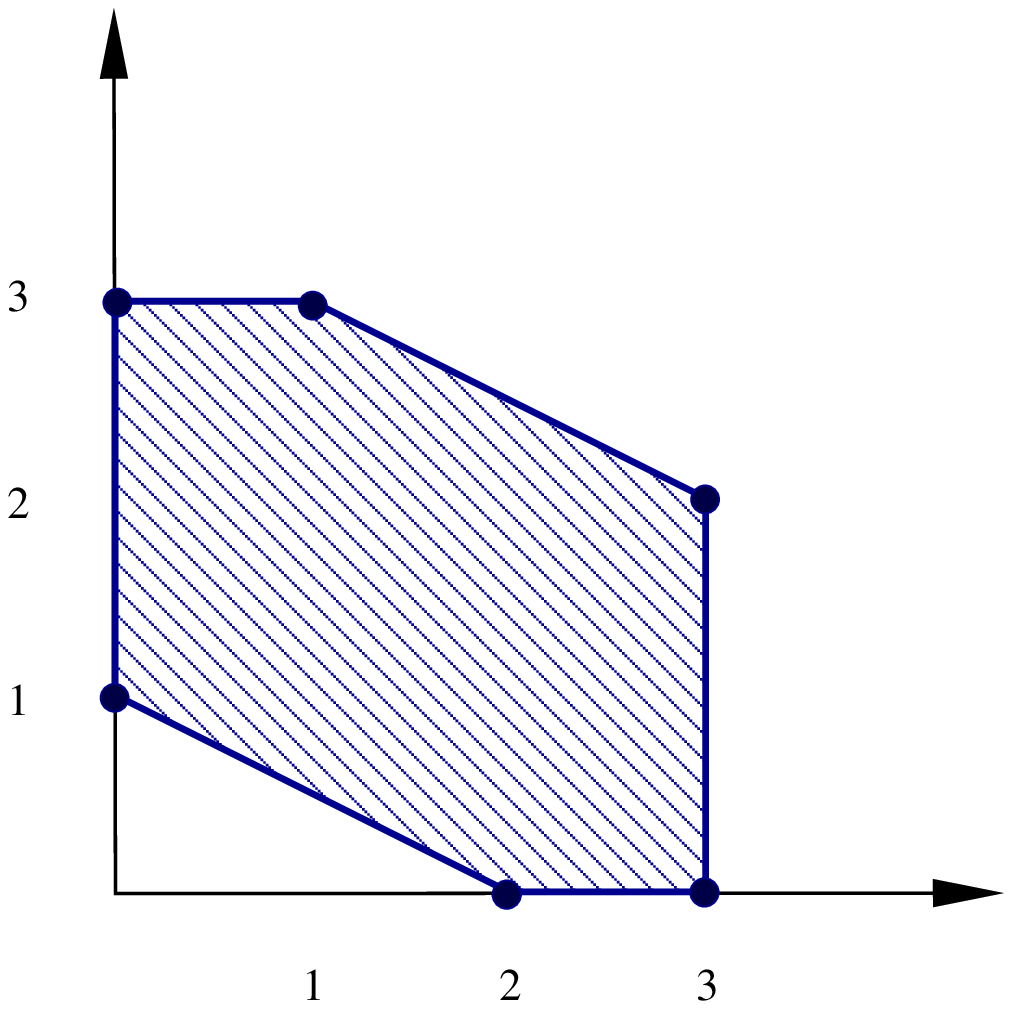}
\caption{The implicit polygons of the curves of example~\ref{Exam:intro}.}
\label{Fig:Nf_exam_intro}
\end{figure}
\end{example}

The paper is organized as follows.
The next section recalls concepts from sparse elimination
and focuses on the Newton polytope of the sparse resultant.
It also defines the problem of computing the implicit polytope.
Section~\ref{S:Rat_ident_denom} solves the problem for rational parameterizations
with identical denominators, by studying relevant mixed subdivisions.
Section~\ref{Spolynomial} determines the implicit polygon
for polynomially parameterized curves.
Section~\ref{S:Rat_dif_denom}
refers to rational parametric curves, where denominators are different.
We conclude with further work in section~\ref{Sfurther}.


\section{Sparse elimination and Implicitization}\label{Spreliminaries}

We first recall some notions of sparse elimination theory;
see~\cite{GKZ} for more information.
Then, we define the problem of implicitization.

Given a polynomial $f$, its \emph{support} $A(f)$ is the set of the exponent vectors
corresponding to monomials with nonzero coefficients.
Its \emph{Newton polytope} $N(f)$ is the convex hull of $A(f)$, denoted CH$(A(f))$.
The \emph{Minkowski sum} $A+B$ of (convex polytopes) $A,B \subset \R^n$ is the set
$A+B=\{a+b~|~ a \in A, b\in B\} \subset \R^n.$

\begin{defn}
Consider Laurent polynomials $f_i, ~i=0,\dots,n$, in $n$ variables, with fixed supports.
Let $\c=(c_{0,0},\ldots,c_{0,s_0},\ldots,c_{n,0},\ldots,c_{n,s_n})$ be the
vector of all nonzero (symbolic) coefficients.
The sparse (or toric) resultant $\RR$ of the $f_i$
is the unique, up to sign, irreducible polynomial in $\Z[\c]$, which vanishes iff
the $f_i$ have a common root in the toric variety corresponding to the supports of the $f_i$.
\end{defn}

Let the system's Newton polytopes be $P_0,\ldots,P_n\subset \R^n$.
Their {\em mixed volume} is the unique integer-valued function, which is symmetric,
multilinear with respect to Minkowski addition, and satisfies MV$(Q,\dots,Q)=n!$ Vol$(Q)$,
for any lattice polytope $Q\subset\R^n$, where Vol$(\cdot)$ indicates Euclidean volume.
For the rest of the paper we assume that the Minkowski sum
$P=P_0+\cdots+P_n \subset \R^n$ is $n$-dimensional.
The family of supports $A_0,\dots,A_n$ is \emph{essential} according
to the terminology of \cite[sec.~1]{Stu94}.
This is equivalent to the existence of a non-zero partial mixed volume
MV$_i=$ MV$(A_0,\ldots,A_{i-1},A_{i+1},\ldots,A_n)$, for some $i\in \{0,\dots,n\}$

A \emph{Minkowski cell} of $P$ is any full-dimensional convex polytope $B=\sum_{i=0}^{n} B_i$,
where each $B_i$ is a convex polytope with vertices in $A_i$.
We say that two Minkowski cells  $B=\sum_{i=0}^{n} B_i$ and $B'=\sum_{i=0}^{n} B_i'$
intersect properly when the intersection of the polytopes $B_i$ and $B_i'$
is a face of both and their Minkowski sum descriptions are compatible, cf.\ \cite{Santos05}.

\begin{defn} \cite[definition~1.1]{Santos05} 
A \emph{mixed subdivision} of $P$ is any family $S$ of
Minkowski cells which partition $P$ and intersect properly as Minkowski sums.
Cell $R$ is \emph{mixed}, in particular \emph{$i$-mixed} or \emph{$v_i$-mixed},
if it is the Minkowski sum of $n$ one-dimensional segments $E_j \subset P_j$,
which are called edge summands, and one vertex $v_i \in P_i$.
\end{defn}

Note that mixed subdivisions contain {\em faces} of all dimensions between~0 and $n$,
the maximum dimension corresponding to cells.
Every face of a mixed subdivision of $P$ has a unique description as
Minkowski sum of subpolytopes of the $P_i$'s.
%
A mixed subdivision is called {\em regular} if it is obtained as the projection of the
lower hull of the Minkowski sum of lifted polytopes
$\{(p_i,\omega_i(p_i))~|~ p_i \in P_i\}$.
If the lifting function $\omega:=\{\omega_i\ldots,\omega_n\}$ is sufficiently generic,
then the induced mixed subdivision is called \emph{tight}.

A monomial of the sparse resultant is called \emph{extreme} if its exponent vector
corresponds to a vertex of the Newton polytope $N(\RR)$ of the resultant.
Let  $\omega$ be a sufficiently generic lifting function.
The $\omega$-extreme monomial of $\RR$ is the monomial with exponent vector
that maximizes the inner product with $\omega$; it corresponds to a vertex of $N(\RR)$ with
outer normal vector $\omega$.

\begin{prop}{\rm \cite{Stu94}.} \label{P:Sturmf_extreme}
For every sufficiently generic lifting function $\omega$, we obtain
the $\omega$-extreme monomial of $\RR$, of the form
\begin{equation}\label{Eq:Sturmf_extreme}
\pm \prod_{i=0}^{n} \prod_{R} c_{i,v_{i}}^{\mathrm{Vol}(R)},
\end{equation}
where ${\rm Vol}(R)$ is the Euclidean volume of $R$, the second product is over
all $v_i$-mixed cells $R$ of the regular tight mixed
subdivision of $P$ induced by $\omega$, and $c_{i,v_i}$ is the
coefficient of the monomial of $f_i$ corresponding to vertex $v_i$.
\end{prop}

\begin{cor}\label{C:surjection-mxd-extreme}
There exists a surjection from the mixed cell configurations onto
the set of extreme monomials of the sparse resultant.
\end{cor}

Given supports $A_0,\ldots,A_n$, the Cayley embedding $\kappa$ introduces  a new point set
$$
C:=\kappa\left(A_{0},A_{1},\ldots, A_{n}\right)=\bigcup_{i=0}^{n} (A_{i} \times \{e_{i}\}) \subset \R^{2n},
$$
where $e_{i}$ are an affine basis of  $\R^n$.

\begin{prop} {\rm [The Cayley Trick] \cite{MV99,Santos05}.} \label{P:Cayley_trick}
There exists a bijection between the regular tight mixed subdivisions
of the Minkowski sum $P$ and the regular triangulations of $C$.
\end{prop}

We now consider the general problem of implicitization.
Let $h_0,\ldots,h_n \in \C[t_1,\ldots,t_r]$ be polynomials in parameters $t_i$.
The implicitization problem is to compute the prime ideal $I$ of all polynomials
$\phi \in \C[x_0,\ldots,x_n]$ which satisfy $\phi(h_0,\ldots,h_n) \equiv 0$ in
$\C[t_1,\ldots,t_r]$.
We are interested in the case where $r=n$, and generalize $h_i$ to be rational expressions
in $\C(t_1,\ldots,t_n)$.
Then $I=\langle \phi \rangle$ is a principal ideal.
Note that $\phi\in\C[x_0,\dots,x_n]$ is uniquely defined up to sign.
The $x_i$ are called implicit variables, $A(\phi)$ is the \emph{implicit support}
and $N(\phi)$ is the implicit polytope.
Usually a rational parameterization may be defined by
\begin{equation}\label{Eq:parameterization}
x_i=\frac{P_i(\t)}{Q(\t)}, ~ i=0,\ldots,n,~~
\gcd(P_0(\t),\ldots,P_n(\t),Q(\t))=1,~ \t=(t_1,\ldots,t_n) .
\end{equation}
Alternatively, the input may be
\begin{equation}\label{Eq:parameterization2}
x_i=\frac{P_i(\t)}{Q_i(\t)}, ~ i=0,\ldots,n,~~ \gcd(P_i(\t),Q_i(\t))=1,~ \t=(t_1,\ldots,t_n).
\end{equation}
In both cases, all polynomial have fixed supports.
We assume that the {\em degree} of the parameterization equals~1.
This avoids, e.g., having all terms in $t^a$ for some $a>1$.

\begin{prop}\label{P:Degree_bounds}
Consider system~(\ref{Eq:parameterization}) and
let $S\subset\Z^n$ be the union of the supports of polynomials $x_iQ-P_i$.
Then, the total degree of the implicit equation $\phi$ is bounded by
the volume of the convex hull CH$(S)$ multiplied by  $n!$.
The degree of $\phi$ in $x_j$ is bounded by the mixed volume
of the $f_i, ~i\neq j$.
\end{prop}

When the rational parameterization is given by equations (\ref{Eq:parameterization2}),
we have the following.

\begin{cor}
Let $S= (A(P_i)+\sum_{j=0,j\ne i}^{n} A(Q_j)) \cup(\sum_{i=0}^n A(Q_i))$.
The total degree of the implicit equation $\phi$ is bounded by
the volume of the convex hull CH$(S)$ multiplied by  $n!$.
\end{cor}

The implicit supports predicted solely by degree bounds are typically larger than optimal.

\section{Rational parameterizations with equal denominators}\label{S:Rat_ident_denom}

We study rationally parameterized curves, when both denominators are the same.
\begin{equation}\label{Eq:rat_param}
x=\frac{P_0(t)}{Q(t)}, ~ y=\frac{P_1(t)}{Q(t)},~~ \gcd(P_i(t),Q(t))=1,\,
P_i,Q\in\C[t], \, i=0,1,
\end{equation}
where the $P_i, Q$ have fixed supports and generic coefficients.
If some $P_i(t),Q(t)$ have a nontrivial GCD, then common terms are divided out and
the problem reduces to the case of different denominators.
In general, the $P_i, Q$ are Laurent polynomials, but this case can be reduced
to the case of polynomials by shifting the supports.

Applying the methods for the case of different denominators
does not lead to optimal implicit support.
The reason is that this does not exploit the fact that the coefficients of $Q(t)$
are the same in the polynomials $xQ-P_0,yQ-P_1$.
Therefore, we introduce a new variable $r$ and consider the following system
\begin{equation}\label{Eq:rat_extra_var}
f_0=x r-P_0(t), \; f_1= y r-P_1(t), \; f_2=r-Q(t) \in \C[t,r] .
\end{equation}
By eliminating $t,r$
the resultant gives, for generic coefficients, the implicit equation in $x,y$.
This is the de-homogenization of the resultant of $x_0-P_0^h , x_1-P_1^h , x_2-Q^h$,
where $P_0^h,P_1^h,Q^h$ are the homogenizations of $P_0,P_1,Q$.
This resultant is homogeneous in $x_0,x_1,x_2$ and generically equals the implicit equation
$\Phi\in\C[x_0,x_1,x_2]$ of parameterization
$\P\rightarrow\P^2 : (t:t_0)\mapsto (x_0:x_1:x_2)=(P_0^h:P_1^h:Q^h)$.

Let the input Newton segments be
$$
B_i=N(P_i), i=0,1,\; B_2=N(Q), \mbox{ where } B_i=\{b_{iL},\dots,b_{iR}\}, i=0,1,2,
$$
where $b_{iL},~b_{iR}$ are the endpoints of segment $B_i$.
The supports of the $f_i$ are
$$
A_0=\{a_{00},a_{0L},\dots,a_{0R}\}, A_1=\{a_{10},a_{1L},\dots,a_{1R}\},
A_2=\{a_{20},a_{2L},\dots, a_{2R} \} \in \N^2 ,
$$
where
\begin{itemize}
\item
each point $a_{i0}=(0,1)$, for $i=0,1,2$, corresponds to the unique term in $f_i$
which depends on $r$,
\item
each other point $a_{it}$, for $t\ne 0$, is of the form $(b_{it},0)$, for one $b_{it}\in B_i$.
\end{itemize}
One could think that index $L=1$ whereas each $R$ equals the cardinality of the
respective $B_i$.
By the above hypotheses either $A_2$ or both $A_0,A_1$ contain $(0,0)$.

\begin{lem}\label{Lmv1mv2}
MV$_{\R}(B_i\cup B_j) =$ MV$_{\R^2}(A_i,A_j),~ i,j \in \{0,1,2\}$,
where MV$_{\R^d}$ denotes mixed volume in $\R^d$.

\begin{proof}
Let $B_i= [ m_i,l_i ],~B_j= [ m_j,l_j ]$ be intervals in $\N$.
If $m_i\le m_j$ and $l_i\le l_j$, then MV$_{\R}(B_i\cup B_j) = l_j-m_i$.
Consider a mixed subdivision of $A_i+A_j$,
with unique mixed cell $( (0,1),(m_i,0)) + ( (0,1),(l_j,0))$,
hence MV$_{\R^2}(A_i,A_j)=l_j-m_i.$
If $m_i\le m_j\le l_j\le l_i$, then MV$_{\R}(B_i\cup B_j) = l_i-m_i$, and
a similar subdivision as above yields a unique mixed cell with this volume.
The rest of the cases are symmetric.
\end{proof}
\end{lem}

Now, let $u=\max\{{b_{0R}, b_{1R}, b_{2R}}\}$.
Let $C_i = $CH$(A_i)$ and consider
the mixed subdivisions of $C = C_0 + C_1 + C_2$.
The following points lie on the boundary of $C$:
$(0,3), (0,2), (u,2), (b_{0L}+b_{1L}+b_{2L},0)$ and $(b_{0R}+b_{1R}+b_{2R},0)$.

The vertices of implicit Newton polytope $N(\Phi)$
correspond to monomials in $x_0,x_1,x_2$;
the power of each $x_i$ is determined by the volumes
of $a_{i0}$-mixed (or simply $i$-mixed) cells, for $i=0,1,2$.
This leads us to computing mixed subdivisions of three polygons in the plane.

\begin{lem}[Cell types]\label{Ltypes}
In any mixed subdivision of $C$, the $i$-mixed cells, with vertex summand $a_{i0}$,
for some $i\in\{0,1,2\}$, have an edge summand $(a_{j0},a_{jh}),\ i\ne j,\ h>0$.
Their second edge summand is from $B_l$, where $\{i,j,l\}=\{0,1,2\}$ and
classifies the $i$-mixed cells in two types:
\\ (I)
If it is $(a_{l0},a_{lm})$, where $a_{lm} =(b_{lm},0)$,
then the cell vertices are $(0,3),(b_{jh},2)$, $(b_{lm},2)$, $(b_{jh} + b_{lm},1)$.
\\ (II)
If it is $(a_{lt},a_{lm})$, where $a_{lt}=(b_{lt},0), a_{lm}=(b_{lm},0)$, then the cell
vertices are $(b_{lt},2), (b_{lm},2)$, $(b_{jh} + b_{lt},1), (b_{jh} + b_{lm},1)$.
\begin{proof}
Any mixed cell has two non-parallel edge summands,  hence one of the edges
is $(a_{j0},a_{jh})$ for some $i\ne j,\ h>0$.
The rest of the statements are straightforward.
\end{proof}
\end{lem}

Observe that for every type-II cell,
there is a non-mixed cell with vertices $(0,3),(b_{lt},2), (b_{lm},2)$.

\begin{example} \label{Exam:folium_section1}
We consider the folium of Descartes:
$$
x=\frac{3t^2}{1+t^3}, y=\frac{3t}{1+t^3} \Rightarrow \phi=x^3 + y^3 -3xy=0.
$$
Now $f_0=xr-3t^2, f_1=yr-3t, f_2=r-(t^3+1)$.
Figure~\ref{Fig:desc_subd} shows the Newton polygons, $C$ and two mixed subdivisions.
The shaded triangle is the only unmixed cell with nonzero area; it is a copy of $C_2$.
The first subdivision shows two cells of type~I, of area~1 and~2,
which yield factors $x$ and $y^2$ respectively, to give term $xy^2$.
The second subdivision has one cell of type~II and area~3, which yields term $x^3$.
We shall obtain an optimal support in example~\ref{Exam:folium_section}.
Now, $u=3$ which equals the total degree of $\phi$.

\begin{figure}[ht]
\centering\includegraphics[width=0.8\textwidth]{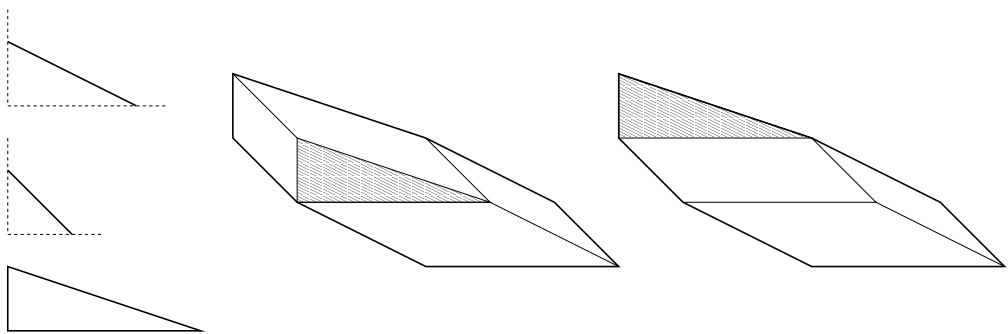}
\caption{Example~\protect\ref{Exam:folium_section1}: polygons $C_i$, and two
mixed subdivisions of $C$.  \label{Fig:desc_subd} }
\end{figure}
\end{example}

Consider segment $L$ defined by vertices $(0,2),(u,2)$ in $C$.

\begin{lem}\label{Lresdeg}
The resultant of the $f_i\in\C[t,r]$
is homogeneous, of degree $u$, wrt the coefficients of the $a_{i0}$, for $i=0,1,2$.

\begin{proof}
Consider any mixed subdivision of $C$ and the cells of type~I and~II.
Consider these cells as closed polygons:
We claim that their union contains segment $L$.
Then, it is easy to see that the total volume of these cells equals $u$.

Consider the closed cells that intersect $L$.
If the intersection lies in the cell interior, then it is a parallelogram, hence
it is mixed and its vertex summand is $(0,1)$, thus it is of type~I.
If the intersection is a cell edge, say $(a_{kl},a_{km})$, for $k\in\{0,1,2\}$
and $1\le l<m$, then the cell above $L$ is unmixed, namely
a triangle with basis $(a_{kl},a_{km})$ and apex at $(0,3)$.
In this case, the cell below $L$ is mixed of type~II.
\end{proof}
\end{lem}

Generically, $u$ equals the total degree of every term in the implicit equation $\phi(x,y)$
wrt $x,y$ and the coefficient of $r$ in $f_2$.
By prop.~\ref{P:Degree_bounds}, the degree of $\Phi(x_0,x_1,x_2)$ is $u$.

In the following, we focus on segment $L$ and subsegments defined
by points $(b_{it},2)\in L, i\in\{0,1,2\}$.
Usually, we shall omit the ordinate, so the
corresponding segments will be denoted by $[b_{jt},b_{kl}]$.
We say that such a segment contributes to some coordinate $e_i$ when a
$i$-mixed cell of the mixed subdivision contains this segment.
Moreover,
\begin{itemize}
\item a type-I, $i$-mixed cell~$a_{i0}+(a_{j0},a_{jt})+(a_{k0},a_{kl})$
is identified with segment $[b_{jt},b_{kl}]$.
\item a type-II, $i$-mixed cell~$a_{i0}+(a_{jt},a_{js})+(a_{k0},a_{kl})$
is identified  with segment $[b_{jt},b_{js}]$ and the coordinate
$e_i$ to which it contributes.
\end{itemize}

We show that one needs to examine only
subsegments defined by endpoints $b_{iL},b_{iR}\in B_i$.
This is equivalent to saying that it suffices to consider mixed subdivisions
induced by linear liftings.

\begin{thm}\label{Tnoint}
Let $S$ be a mixed subdivision of $C_0+C_1+C_2$, where an internal point $b_i \in B_i$
defines a 0-dimensional face $(b_i,2)=(b_i,0)+(0,1)+(0,1) \in L$.
Then, the point of $N(\phi)$ obtained by $S$ cannot be a vertex because it is
a convex combination of points obtained by other mixed subdivisions
defined by points of $B_0,B_1,B_2$ which are either endpoints, or are used in
defining $S$ except from $(b_i,2)$.
\end{thm}

The theorem is established by lemmas~\ref{L:II-II}, \ref{L:I-I} and~\ref{L:I-II}.
We shall construct mixed subdivisions that yield points in the $e_ke_j$-plane
whose convex hull contains the initial point.
All cells of the original subdivision which are not mentioned are taken to be fixed,
therefore we can ignore their contribution to $e_k,e_j$.
All convex combinations in these lemmas are decided by the $3\times 3$
orientation determinant (cf.\ expression~(\ref{EQccw2a})).

\begin{lem}[II-II]\label{L:II-II}
Consider the setting of theorem~\ref{Tnoint} and suppose that $(b_i,2)$ is
a vertex of two adjacent type~II cells.
Then, the theorem follows.
\begin{proof}
If both cells are $j$-mixed, then the same point in $e_ke_j$-plane is obtained
by one $j$-mixed cell equal to their union, $\{i,j,k\}=\{0,1,2\}$.
If the cells are $j$- and $k$-mixed, then there are two mixed subdivisions
yielding points in the $e_ke_j$-plane, which define a segment that contains
the initial point.
The subdivisions have one $j$-mixed or one $k$-mixed cell respectively,
intersecting the entire subsegment.
\end{proof}
\end{lem}

\begin{lem}[I-I]\label{L:I-I}
Consider the setting of theorem~\ref{Tnoint} and suppose that $(b_i,2)$ is
a vertex of two adjacent type~I cells.
Wlog, these are $k$- and $j$-mixed cells, $\{i,j,k\}=\{0,1,2\}$.
Then, the theorem follows.
\begin{proof}
Let $[b_{jl},b_i], [b_i,b_{kt}]$ be the subsegments defined on $L$ by
the two mixed cells, and let $\alpha,\beta$ be their respective lengths.
Since $b_i$ is internal, $b_{iR}$ lies to its right-hand side and $b_{iL}$ lies to its left-hand side.

Case $b_{iR} < b_{kt}$ and $b_{iL} > b_{jl}$. Let $\gamma=b_i-b_{iL}$ and $\delta=b_{iR}-b_i$.
The initial point $(\alpha,\beta)$ shall be enclosed by two points.
The mixed subdivision with type-I cells corresponding to  $[b_{jl},b_{iR}]$ and
$[b_{iR},b_{kt}]$ yields point $(\alpha + \delta, \beta - \delta)$.
The subdivision with type-I cells corresponding to
$[b_{jl},b_{iL}], [b_{iL},b_{kt}]$ yields point $(\alpha - \gamma, \beta + \gamma)$.

Case $b_{iR} < b_{kt}$ and $b_{iL} < b_{jl}$. Let $\gamma=b_{jl}-b_{iL}$ and
$\delta=b_{iR}-b_i < \beta$. The initial point is
$(\alpha+v_k,\beta+v_j)$, where $v_k,v_j\ge 0$ is the contribution
to $e_k,e_j$ respectively from subsegment $[b_{iL},b_{jl}]$, and
$v_k+v_j \le \gamma$.
Now consider~3 mixed subdivisions on $[b_{iL},b_{kt}]$:
The first containing the type-II $k$-mixed cell $[b_{iL},b_{iR}]$
and the type-I $j$-mixed cell $[b_{iR},b_{kt}]$
gives point $(\alpha+\gamma+\delta, \beta - \delta)$.
The second containing the type-I $j$-mixed cell $[b_{iL},b_{kt}]$
gives point $(0, \alpha+\beta+\gamma)$.
The third containing the type-I $i$-mixed cell $[b_{jl},b_{kt}]$
and the initial cells in $[b_{iL},b_{jl}]$, gives $(v_k,v_j)$.

Case $b_{iR} > b_{kt}$ and $b_{iL} > b_{jl}$. Let $\gamma=b_i-b_{iL}<\alpha$ and
$\delta=b_{iR}-b_{kt}$. The initial point is
$(\alpha+v_k,\beta+v_j)$, where $v_k,v_j\ge 0$ is the contribution
to $e_k,e_j$ respectively from $[b_{kt},b_{iR}]$, and
$v_k+v_j \le  \delta$.
Now consider~3 mixed subdivisions on $[b_{jl},b_{iR}]$:
The first containing the type-I $i$-mixed cell $[b_{jl},b_{kt}]$
and the initial cells in $[b_{kt},b_{iR}]$,
gives point $(v_k,v_j)$.
The second containing the type-I $k$-mixed cell $[b_{jl},b_{iR}]$,
gives point $(\alpha+\beta+\delta,0)$.
The third containing the type-I $k$-mixed cell $[b_{jl},b_{iL}]$
and the type-II $j$-mixed cell $[b_{iL},b_{iR}]$, gives $(\alpha - \gamma,\beta+\gamma+\delta)$.

Case $b_{iR} > b_{kt}$ and $b_{iL} < b_{jl}$. Let $\gamma=b_{jl}-b_{iL}$ and
$\delta=b_{iR}-b_{kt}$. The initial point is
$(\alpha+v_k+u_k,\beta+v_j+u_j)$, where $v_k,v_j\ge 0$ is the contribution
to $e_k,e_j$ respectively from $[b_{kt},b_{iR}]$, and
$v_k+v_j \le  \delta$. Similarly, $u_k,u_j\ge 0$ is the contribution
to $e_k,e_j$ respectively from $[b_{iL},b_{jl}]$, and
$u_k+u_j \le  \gamma$.
Now consider~3 mixed subdivisions on $[b_{iL},b_{iR}]$:
The first containing the type-II $k$-mixed cell $[b_{iL},b_{iR}]$,
gives point $(\alpha + \beta+\gamma+\delta,0)$.
The second containing the type-II $j$-mixed cell $[b_{iL},b_{iR}]$,
gives point $(0, \alpha + \beta+\gamma+\delta)$.
The third containing the type-I $i$-mixed cell $[b_{jl},b_{kt}]$
and the initial cells in $[b_{iL},b_{jl}]$ and $[b_{kt},b_{iR}]$,
 gives point $(v_k + u_k, v_j + u_j)$.
\end{proof}
\end{lem}
\begin{figure}[ht]
\psfrag{S}{\footnotesize $S$}
\psfrag{S1}{\footnotesize $S_1$}
\psfrag{S2}{\footnotesize $S_2$}
\psfrag{S3}{\footnotesize $S_3$}
\psfrag{a+g+d}{\footnotesize $\alpha+\gamma+\delta$}
\psfrag{u}{\footnotesize $\alpha+\beta+\gamma $}
\psfrag{b-d}{\footnotesize $\beta-\delta$}
\psfrag{vk}{\footnotesize $v_k$}
\psfrag{vj}{\footnotesize $v_j$}
\psfrag{vk,vj}{\footnotesize $v_k,v_j$}
\psfrag{a}{\footnotesize $\alpha$}
\psfrag{b}{\footnotesize $\beta$}
\psfrag{g}{\footnotesize $\gamma$}
\psfrag{d}{\footnotesize $\delta$}
\psfrag{b+g}{\footnotesize $\beta+\gamma$}
\psfrag{a+g}{\footnotesize $\alpha+\gamma$}
\psfrag{ek}{\footnotesize $e_k$}
\psfrag{ej}{\footnotesize $e_j$}
\psfrag{ei}{\footnotesize $e_i$}
\psfrag{biL}{\footnotesize $b_{iL}$}
\psfrag{bjl}{\footnotesize $b_{jl}$}
\psfrag{bi}{\footnotesize $b_i$}
\psfrag{biR}{\footnotesize $b_{iR}$}
\psfrag{bkt}{\footnotesize $b_{kt}$}
\centering\includegraphics[width=0.55\textwidth]{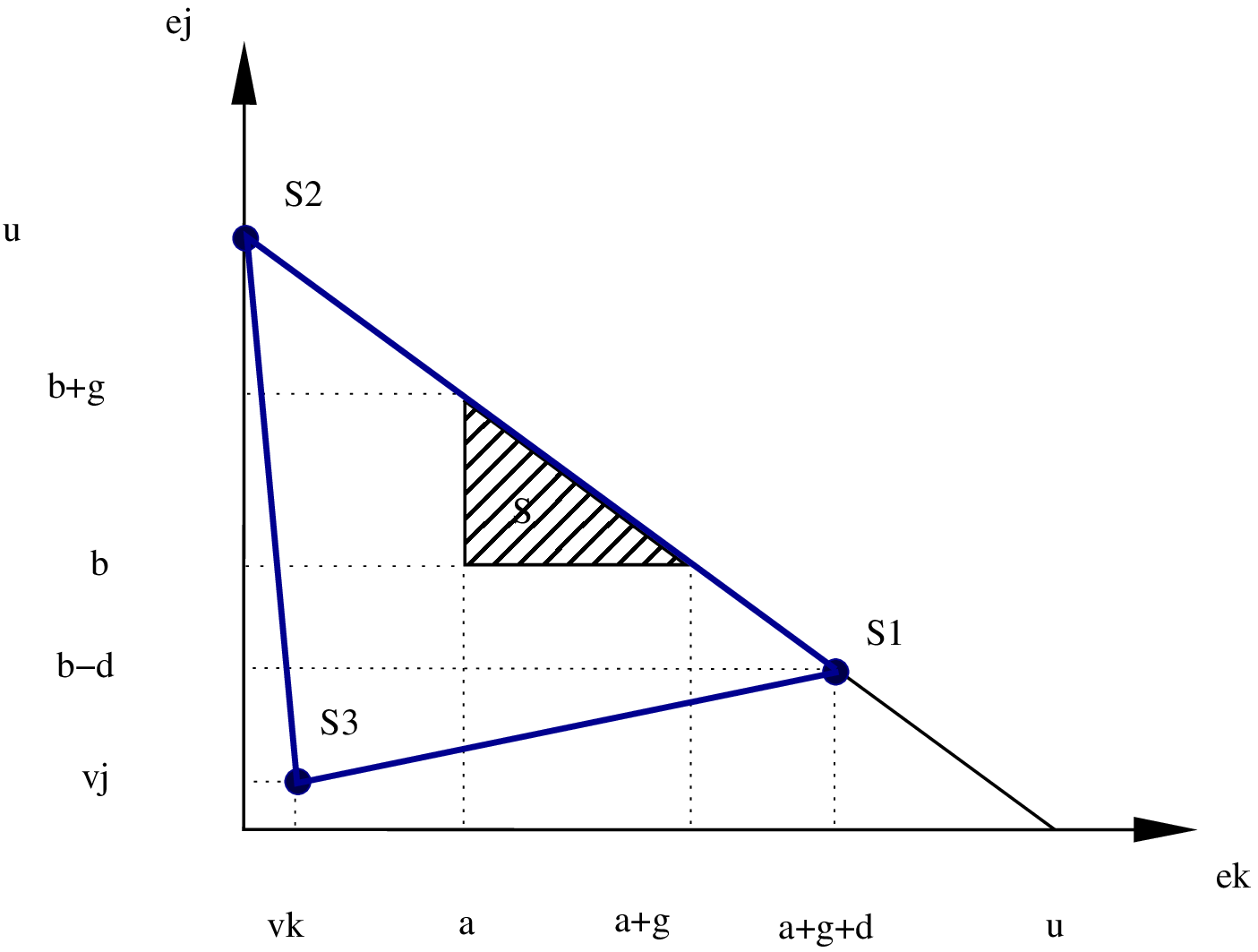}
\hskip3em
\includegraphics[width=0.3\textwidth]{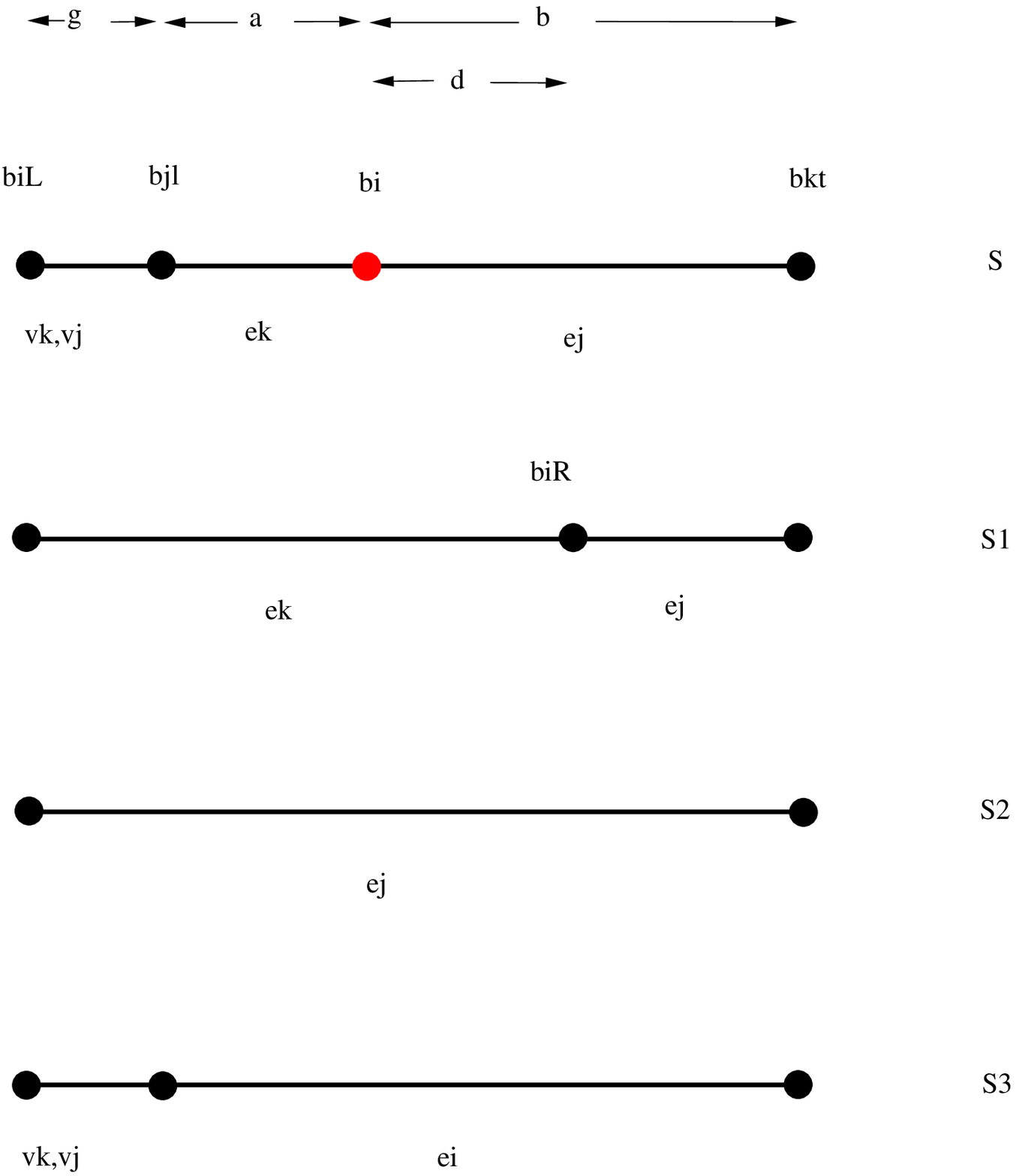}
\caption{The three points that enclose the point given by $S$
and the corresponding mixed subdivisions for the second case of Lemma~\protect\ref{L:I-I}.
\label{Fig:LI-I} }
\end{figure}

\begin{lem}[I-II]\label{L:I-II}
Consider the setting of theorem~\ref{Tnoint} and suppose that $(b_i,2)$ is
a vertex of two adjacent type~II and~I cells.
Wlog, these are $k$- and $j$-mixed cells, $\{i,j,k\}=\{0,1,2\}$.
Then, the theorem follows.
\begin{proof}
Let $[b_{il},b_{i}], [b_i,b_{kt}]$ be the subsegments defined on $L$ by
the two mixed cells, and let $\alpha,\beta$ be their respective lengths.
Since $b_i$ is internal, $b_{iR}$ lies to its right-hand side.
Moreover, the initial $k$-mixed cell implies the existence of 1-dimensional face
$(b_{i},2)+a_{k0}+E_{jl}$, for some edge $E_{jl}=(a_{j0},a_{jl}) \subset B_j$.
The initial $j$-mixed cell implies the existence of 1-face
$(b_{i},2)+a_{j0}+E_{kt}$, for edge $E_{kt}=(a_{k0},a_{kt})\subset B_k$.
The second 1-face cannot be to the left of the first one, hence $b_{jl}\le b_{kt}$.
Hence, $b_{jL}\le b_{kt}$.

Case $b_{iR}\le b_{kt}$.
The initial point $(\alpha,\beta)$ shall be enclosed by two points.
The mixed subdivision with type-I cell $[b_{il},b_{kt}]$
yields point $(0,\alpha + \beta)$.
The subdivision with type-II and type-I cells corresponding to
$[b_{il},b_{iR}], [b_{iR},b_{kt}]$ sets $e_k> \alpha, e_j< \beta$, where
$e_k+e_j=  \alpha + \beta$.

Case $b_{iR} > b_{kt}$ and $b_{jL} > b_{il}$.
Consider subsegment $[b_{il},b_{iR}]$:
the initial point is $(\alpha+v_k,\beta+v_j)$, where $v_k,v_j\ge 0$ is the contribution
to $e_k,e_j$ respectively from subsegment $[b_{kt},b_{iR}]$, and
$v_k+v_j\le \gamma= b_{iR}-b_{kt}$.
Now consider~3 mixed subdivisions on $[b_{il},b_{iR}]$:
One $k$-mixed cell $[b_{il},b_{iR}]$ gives point $(\alpha+\beta+\gamma,0)$.
One $j$-mixed cell $[b_{il},b_{kt}]$
and the initial cells in $[b_{kt},b_{iR}]$ give $(v_k,\alpha+\beta+v_j)$.
One $k$-mixed cell $[b_{il},b_{jL}]$, one $i$-mixed cell $[b_{jL},b_{kt}]$
and the initial cells in $[b_{kt},b_{iR}]$ give $(e_k+v_k,v_j)$, for some
$e_k\le \alpha+\beta$.

Case $b_{iR} > b_{kt}$ and $b_{jL} \le b_{il}$.
Consider subsegment $[b_{jL},b_{iR}]$:
the initial point is $(\alpha+u_k+v_k,\beta+u_j+v_j)$, where $v_k,v_j$ are as above,
$u_k,u_j\ge 0$ correspond to subsegment $[b_{jL},b_{il}]$,
and $u_k+ u_j\le \delta = b_{il}-b_{jL}$.
Now consider~3 mixed subdivisions on $[b_{jL},b_{iR}]$:
One $k$-mixed cell $[b_{jL},b_{iR}]$ gives point $(\alpha+\beta+\gamma+\delta,0)$.
One $j$-mixed cell $[b_{il},b_{iR}]$
and the initial cells in $[b_{jL},b_{il}]$ give $(u_k,\alpha+\beta+\gamma+u_j)$.
One $i$-mixed cell $[b_{jL},b_{kt}]$,
and the initial cells in $[b_{kt},b_{iR}]$ give $(v_k,v_j)$.
\end{proof}
\end{lem}

In the next lemma and corollary, we shall determine certain points in $N(\Phi)$.
We shall later see that among these points lie the vertices of $N(\Phi)$ and,
therefore, the vertices of $N(\phi)$.
Recall that MV$_i = $MV$_{\R^2}(A_j,A_k)$, where $\{i,j,k\}=\{0,1,2\}$.

\begin{lem}\label{Lsame_corners}
Let $b_{tL}= \min\{b_{iL},b_{jL}\},~b_{mR}= \max\{b_{iR},b_{jR}\}$ and
$\Delta = [b_{tL},b_{mR}]$,
for $i\ne j\in\{0,1,2\}$ and $t,m\in\{i,j\}$ not necessarily distinct.
Set $e_{\lambda}= |\Delta|$, where $\lambda\in\{0,1,2\}-\{i,j\}$, and $e_i=e_j=0$.
Then, add $b_{tL}$ to $e_{\tau}$, where $\tau\in\{i,j\}-\{t\}$,
and add $u-b_{mR}$ to $e_{\mu}$, where $\mu\in\{i,j\}-\{ m \}$.
Then, $(e_0,e_1)$ is a vertex of $N(\phi)$.

\begin{proof}
Clearly $\Delta=$CH$(B_i\cup B_j) \subseteq[0,u]$, so MV$_{\lambda}=|\Delta|$.
It is possible to construct a mixed subdivision that yields the implicit vertex.
If $t \neq m$, then the mixed subdivision contains a type-I mixed cell
$(a_{t0},a_{tL})+(a_{m0},a_{mR_m})+a_{\lambda 0}$ which
intersects segment $L$ at subsegment $[b_{tL},b_{mR}]$.
This contributes $\mbox{\rm MV}_{\lambda}=b_{mR}-b_{tL}$ to $e_{\lambda}$.
There is a type-I cell $(a_{\lambda 0},a_{\lambda L})+(a_{t0},a_{tL})+a_{\tau 0}$ which
intersects $L$ at subsegment $[0,b_{tL}]$.
This contributes $b_{tL}$ to $e_{\tau}$.
Similarly, we assign the area $u-b_{mR}$ of the type-I cell
$(a_{\lambda 0},a_{\lambda R})+(a_{m0},a_{mR_m})+a_{\mu 0}$ to $e_{\mu}$.

If $t=m$, then  $\Delta$ is an edge of one of the initial Newton segments,
say $B_t$, and $\Delta=[b_{tL},b_{tR}]$.
The mixed subdivision contains the type-II mixed cell
$(a_{\tau 0},a_{\tau L})+(a_{tL},a_{t R})+a_{\lambda 0}$ which contributes
$\mbox{\rm MV}_{\lambda}=|\Delta|=b_{tR}-b_{tL}$ to $e_{\lambda}$.
There are also two type-I cells intersecting $L$ at its leftmost and
rightmost subsegments, as in the previous case.
Since $t=m$, we have $\mu=\tau$, hence $e_t=0$.

The type-I mixed cells in any of the above mixed subdivisions vanish
when $b_{tL}=0$ or $b_{mR}=u$.
Notice that $e_i+e_j+e_{\lambda}=u$ and since $e_{\lambda}$ is maximized, $(e_0,e_1,e_2)$
defines a vertex of $N(\Phi)\subset\R^3$.
Projecting to the $e_0e_1$-plane yields the implicit vertex.

\begin{figure}[ht]
\psfrag{Al}{\footnotesize $A_{\lambda}$}
\psfrag{Ai}{\footnotesize $A_i$}
\psfrag{Aj}{\footnotesize $A_j$}
\psfrag{t=m=j}{\footnotesize $t=m=j$}
\psfrag{tn=m}{\footnotesize $i=t\neq m=j$}
\centering\includegraphics[width=0.9\textwidth]{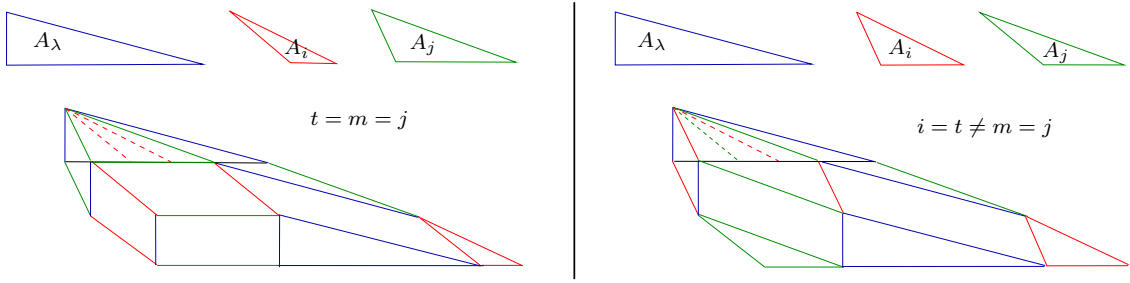}
\caption{Lemma~\protect\ref{Lsame_corners}: the
mixed subdivisions for the cases $t=m$ and $t \neq m$.  \label{Fig:L419} }
\end{figure}
\end{proof}
\end{lem}

The following corollary is proven similarly to the above proof.

\begin{cor}\label{C_C0}
Under the notation of lemma~\ref{Lsame_corners} consider the following~3 definitions:
\begin{enumerate}\item
$b_{tL}= \min\{b_{iL},b_{jL}\},~b_{mR}= \min\{b_{iR},b_{jR}\}$,
\item
$b_{tL}=\max\{b_{iL},b_{jL}\},~b_{mR}= \max\{b_{iR},b_{jR}\}$,
\item
$b_{tL}= \max\{b_{iL},b_{jL}\},~b_{mR}= \min\{b_{iR},b_{jR}\}$,
provided $b_{tL} \le b_{mR}$.
\end{enumerate}
In each case, define $e_0,e_1,e_2$ as in lemma~\ref{Lsame_corners}.
Then $(e_0,e_1) \in N(\phi)$.
\end{cor}

\subsection{The implicit vertices}

Overall, there are three cases for the relative positions of the $B_i$:
\begin{enumerate} \item
CH$(B_i\cup B_j)=[0,u]$ for all pairs $i,j$.
\item
CH$(B_j\cup B_l) =$ CH$(B_i\cup B_l) =[0,u]\ne$ CH$(B_i\cup B_j)$.
\item
CH$(B_i\cup B_j)=[0,u]\ne$ CH$(B_l\cup B_t)$ for $t=i,j$.
\end{enumerate}
Orthogonally, we can distinguish the following two cases:
\\ (A) there exists at least one $B_i=[0,u]$,
\\ (B) none of the $B_i$'s equals $[0,u]$.

In case~(B), every union $B_i\cup B_j$ contains either 0 or $u$.
Cases~(1B) and~(3A) cannot exist, which leaves~4 cases overall.
In the sequel, we let $E_{it}$ denote a segment $(a_{i0}, a_{it})\subset B_i$.

\begin{thm}[case~(A)]\label{Tcase2a}
If all unions CH$(B_i\cup B_j)=[0,u],i\ne j$, then
the implicit polygon is a triangle with vertices $(0,0), (0,u), (u,0)$.
If exactly one support, say $B_k,k\in\{0,1,2\}$, equals $[0,u]$, then $N(\phi)$
has up to~5 vertices in the following set of $(e_i,e_j)$ vectors:
$$
\{ (u,0), (0,u), (0,u-b_{iR}+b_{iL}), (b_{jL},u-b_{iR}) , (u-b_{jR}+b_{jL},0) \} ,
$$
where $\{i,j,k\}=\{0,1,2\}$, assuming $i,j$ are chosen so that
\begin{equation}\label{EQassume}
b_{iL}(u-b_{jR})\ge b_{jL}(u-b_{iR}) .
\end{equation}
\begin{figure}[ht]
\psfrag{ei}{$e_i$} \psfrag{ej}{$e_j$}\psfrag{e2}{$e_k$}
\psfrag{biL}{$b_{iL}$} \psfrag{bjL}{$b_{jL}$} \psfrag{biR}{$b_{iR}$}
\psfrag{bjR}{$b_{jR}$} \psfrag{u}{$u$} \psfrag{u_biR}{$u-b_{iR}$}
\psfrag{u_Bi}{$u-b_{iR}+b_{iL}$} \psfrag{u_Bj}{$u-b_{jR}+b_{jL}$}
\centering\includegraphics[width=0.55\textwidth]{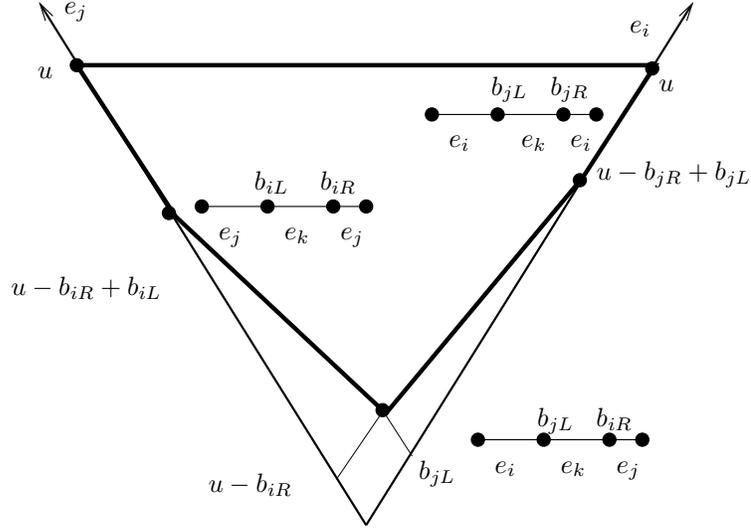}
\caption{The implicit polygon in case~(2A), in the $e_ie_j$-plane,
and the subdivisions of the proof of
theorem~\ref{Tcase2a}.}\label{Fig:case2a}
\end{figure}
\begin{proof}
First is the case~(1A), established from lemma~\ref{Lsame_corners}.
The second statement concerns case~(2A):
By switching $i$ and $j$, assumption~(\ref{EQassume}) can always be satisfied.
Unless $B_i\subset B_j$ or $B_j\subset B_i$, this assumption holds simply by
choosing $i,j$ so that $b_{jL}\le b_{iL}$.

The vertices $(u,0), (0,u)$ are obtained by lemma~\ref{Lsame_corners},
applied to CH$(B_j\cup B_k)$ and CH$(B_i\cup B_k)$ respectively.
The third point is obtained by a mixed subdivision with two type-I cells
$E_{iL}+a_{j0}+E_{kL}, E_{iR}+a_{j0}+E_{kR}$, which contribute the lengths of
$[b_{kL},b_{iL}], [b_{iR},b_{kR}]$ to $e_j$,
and one type-II cell $E_{i0}+E_{jt}+a_{k0}$, contributing the length of
$[b_{iL},b_{iR}]$ to $e_k$, where $E_{i0}$
is the horizontal edge of $A_i$ and $t \in\{L,R\}$; see figure~\ref{Fig:case2a}.
By switching $i$ and $j$ we define a subdivision that yields the fifth point.

The fourth point is obtained by a subdivision with~3 type-I cells:
$a_{i0}+E_{jL}+E_{kL}, E_{iR}+a_{j0}+E_{kR}$ and $E_{iR}+E_{jL}+a_{k0}$,
which contribute to $e_i,e_j$ and $e_k$ respectively, see fig.~\ref{Fig:case2a}.
It suffices to show that the line defined by this and the third point
supports the implicit polygon.
An analogous proof then shows that the line defined by this and the~5th
point also supports the polygon, and the theorem follows.
Our claim is equivalent to showing
\begin{equation}\label{EQccw2a}
\det\left[ \begin{array}{ccc}
b_{jR} & u-b_{iR} & 1 \\ 0 & u-b_{iR}+b_{iL} & 1 \\ e_i & e_j & 1 \end{array}\right]
\le 0 \Leftrightarrow b_{iL} (e_{i}-b_{jL}) \ge b_{jL} (u-b_{iR}-e_j) .
\end{equation}
We consider the rightmost subsegment on $L$, where one endpoint is $b_{kR}=u$.
This contributes to either $e_i$ or $e_j$ an amount equal to the length of
a subsegment extending at least as far left as $b_{jR}$ or $b_{iR}$, respectively.
Symmetrically, the leftmost subsegment has endpoint $b_{kL}=0$
and contributes to $e_i$ or $e_j$ the length of a subsegment
extending at least as far right as $b_{jL}$ or $b_{iL}$, respectively.
In general, there are~4 cases, depending on the contribution
of the rightmost and leftmost subsegments.
The last case is infeasible if $B_i,B_j$ have no overlap.

If the rightmost subsegment contributes to $e_j$ then $e_j\ge u-b_{iR}$.
If the leftmost subsegment contributes to $e_j$ then
this contribution is at least $b_{iL}$, hence $e_j\ge u-b_{iR}+b_{iL}$,
where $e_i\ge 0$.
Otherwise, the leftmost subsegment contributes to $e_i$, thus $e_i\ge b_{jL}$.
In both cases, inequality~(\ref{EQccw2a}) follows.

If the rightmost subsegment contributes to $e_i$ then $e_i\ge u-b_{jR}$.
If the leftmost subsegment also contributes to $e_i$, then $e_i\ge u-b_{jR}+b_{jL}$.
Using also $e_j\ge 0$, it suffices to prove $b_{iL}(u-b_{jR})\ge b_{jL}(u-b_{iR})$.
Otherwise, the leftmost subsegment contributes to $e_j$, so $e_j\ge b_{iL}$, and it
suffices to prove $b_{iL}(u-b_{jR}-b_{jL})\ge b_{jL}(u-b_{iR}-b_{iL})$.
Both sufficient conditions are equivalent to assumption~(\ref{EQassume}).
\end{proof}
\end{thm}

\begin{thm}[case~(B)]\label{Tcase2}
If none of the $B_t$'s is equal to $[0,u]$, then we may choose $\{i,j,k\}=\{0,1,2\}$
such that:
$$
0 < b_{iL}\le b_{iR} =u, 0=b_{jL}\le b_{jR}<u, 0 \le b_{kL}\le b_{kR}<u .
$$
Then, $N(\phi)$ has at most~5 or~4 vertices, depending on whether $b_{kL}$ is positive or~0.
In the former case, the vertices $(e_i,e_j)$ lie in
$$
\{ (b_{jR},0), (b_{kR},u-b_{kR}), (b_{kL},u-b_{kL}), (0,u-b_{0L}), (0,0), \}
$$
and, in the latter case, the third and fourth vertices are replaced by $(0,u)$.
\end{thm}

By lemma~\ref{Lresdeg}, at every point $e_k= u-e_i-e_j$.
The theorem is established by the following two lemmas.

\begin{lem}[case~(2B)]\label{Lcase2b}
Suppose $b_{kL}=0$ in theorem~\ref{Tcase2}
and wlog assume $b_{jR}\le b_{kR}$.
Then, $N(\phi)$ has up to~4 vertices $(e_i,e_j)$ in the set
$$
\{ (b_{jR},0), (b_{kR},u-b_{kR}), (0,u), (0,0) \} .
$$
\begin{proof}
The last two vertices follow from lemma~\ref{Lsame_corners}, applied to
CH$(B_i\cup B_k)$ and CH$(B_i\cup B_j)$, respectively.
The same lemma, applied to CH$(B_j\cup B_k)$, yields the second vertex,
whereas cor.~\ref{C_C0}(1) yields the first point.
It suffices to show that any point $(e_i,e_j)\in N(\phi)$ defines a counter-clockwise
turn in the $e_ie_j$-plane, when appended to $(b_{jR},0)$ and $(b_{kR},u-b_{kR})$.
This is equivalent to proving
\begin{equation}\label{EQccw}
\det\left[ \begin{array}{ccc}
b_{jR} & 0 & 1 \\ b_{kR} & u-b_{kR} & 1 \\ e_i & e_j & 1 \end{array}\right]
\ge 0 \Leftrightarrow e_j(b_{kR}-b_{jR}) \ge (u-b_{kR}) (e_i-b_{jR}) .
\end{equation}
Rightmost segment $[b_{kR},b_{iR}=u]$ cannot contribute to $e_i$, since
each corresponding mixed cell has an edge summand from $A_i$.
If the segment lies in a $j$-mixed cell, then $e_j\ge u-b_{kR}$ and $e_i\le b_{kR}$,
and inequality~(\ref{EQccw}) is proven.
Otherwise, at least a subsegment contributes to a $k$-mixed cell.

If this subsegment contains $b_{kR}$, then it must extend at least to the next
endpoint lying left of $b_{kR}$, hence to $b_{jR}$ or $b_{iL}$.
In the latter case, the subsegment to the left of $b_{iL}$ cannot contribute to $e_i$.
Thus, in any case, $e_i\le b_{jR}$, so~(\ref{EQccw}) is proven.

If none of the above happens, then the subsegment contributing to $e_k$ does not
contain $b_{kR}$, so the only way for the $k$-mixed cell to be defined is to have
$b_{iL}$ lie in $(b_{kR},b_{iR})$ and $k$-mixed cell intersecting $L$ at $[b_{iL},b_{iR}]$.
Then, $[b_{kR},b_{iL}]$ contributes to $e_j$, so the $j$-mixed cell
intersects $L$ at $[b_{kt},b_{iL}]$, where $t\in\{L, R \}$.
If $b_{kt}=b_{kL}$, then $e_i=0$ and~(\ref{EQccw}) is proven.

Otherwise, $b_{kt}=b_{kR}$.
The $j$-mixed cell is of type~I and implies that the $1$-dimensional face
$(b_{iL}, 2)+E_{kR}$ belongs to the subdivision, see lemma~\ref{Ltypes}.
The $k$-mixed cell is of type~II, with some edge summand $E_{jt}\subset A_j$, which implies
that the $1$-face $(b_{iL}, 2)+E_{jt}$ is in the subdivision and cannot lie to the
left of the previous $1$-face.
Since $b_{jR}\le b_{kR}$, we have $b_{kR}= b_{jR}$, hence $e_i\le b_{jR}$.
\end{proof}
\end{lem}


\begin{lem}[case~(3B)]\label{Lcase3b}
Suppose $b_{kL} > 0$ in theorem~\ref{Tcase2}.
Then, $N(\phi)$ has up to~5 vertices $(e_i,e_j)$ in the set
$$
\{ (b_{jR},0), (b_{kR},u-b_{kR}), (b_{kL},u-b_{kL}), (0,u-b_{iL}), (0,0) \} .
$$
\begin{proof}
The last vertex follows from lemma~\ref{Lsame_corners}, applied to CH$(B_i\cup B_j)$.
We shall prove that the first two points are vertices; they are
obtained by using CH$(B_j\cup B_k)$ in lemma~\ref{Lsame_corners} and cor.~\ref{C_C0}(1).
Which point is obtained from which lemma depends on the sign of $b_{jL}-b_{kL}$.
The third and fourth vertices are established analogously, by considering
CH$(B_i\cup B_k)$.

Our proof shall establish inequality~(\ref{EQccw}).
If $b_{jR}\le b_{kR}$, this is similar to the proof of lemma~\ref{Lcase2b}.
Otherwise, $b_{kR} < b_{jR}$, and
the rightmost segment $[b_{jR},b_{iR}=u]$ cannot contribute to $e_i$.
If it contributes to $e_k$ only, then $e_k\ge u-b_{jR}$ so $e_i+e_j\le b_{jR}$
and~(\ref{EQccw}) follows.

If it contributes to $e_j$ only, the union of the corresponding $j$-mixed cells
intersect $L$ at a segment with an endpoint to the left of $b_{jR}$,
namely $b_{kt}, t\in\{L,R\}$, or $b_{iL}$.
In the former case,
$e_i\le b_{kR}$ and $e_j\ge u-b_{kR}$.
In the latter case, $[0,b_{iL}]$ contributes to $e_k$ only, so $e_i=0,e_j=u-b_{iL}$.
In both cases,~(\ref{EQccw}) follows readily.

Lastly, $[b_{jR},b_{iR}]$ might be split into subsegments
$[b_{jR},b_{iL}], [b_{iL},b_{iR}]$, contributing to $e_k,e_j$ respectively.
The corresponding cells are of type~I and type~II,
the latter having an edge summand from $A_k$.
This requires the subdivision to have $j$-faces
$(b_{iL},k)+E_{jR}$ and $(b_{iL},2)+E_{kt}, t \in\{L,R\}$,
where the first lies to the left of the second, see lemma~\ref{Ltypes}.
This cannot happen because $b_{kR}<b_{jR}$.
\end{proof}
\end{lem}

\begin{example}\label{Exam:circle}
For the unit circle,
$$
x=2t/(t^2+1), ~y=(1-t^2)/(t^2+1)
$$
we have $f_0=xt^2-2t+x,~ f_1=(y+1)t^2+(y-1)$.
In lemma~\ref{Lsame_corners}, the sets $B_0=\{1\},B_1=\{0,2\},B_2=\{0,2\}$
yield terms $x^2,y^2,1$ in $\phi$ and, hence, an optimal support.
See example~\ref{Exam:circle_diff} for a treatment assuming different denominators.
\end{example}


\begin{example}\label{Exam:folium_section}
For the folium of Descartes 
$$
x=\frac{3t^2}{t^3+1} , y=\frac{3t}{t^3+1} \Rightarrow \phi=x^3 + y^3 -3xy=0 ;
$$
see example~\ref{Exam:folium_section1} and figure~\ref{Fig:desc_subd}.
Now, $B_0=\{2\},B_1=\{1\},B_2=\{0,3\}$, hence this is case~(2B).
In theorem~\ref{Tcase2a}, we set $i=0,j=1,k=2$ and obtain,
in the order stated by the theorem: $x^3, x^3, y^3, y^3, xy$, hence an optimal support.

If we do not account for the same denominators,
use degree bounds alone,
or project the Sylvester resultant, we obtain an overestimation of the support.
\end{example}

\begin{example}\label{Exam:rat_param_sameden}
$$
x=\frac{2t^3+t+1}{t^2+1} , y=\frac{t^4+t^3-1}{t^2+1} ,
$$
hence $B_0=\{0,1,3\},B_1=\{0,3,4\},B_2=\{0,2\}$, so this is case~(2A) with $B_1=[0,u]$.
In theorem~\ref{Tcase2a}, we set $i=0,j=2$ and obtain the vectors
$(e_i,e_j)=(4,0),(0,4),(0,1),(0,3),(2,0)$, in the order stated by the theorem.
This yields the implicit points $(e_0,e_1)=(4,0),(0,0),(0,3),(0,1),(2,2)$,
hence vertices $(4,0), (0,0),(0,3), (2,2)$.
These define the optimal polygon because the implicit equation is
$$
\phi=59-21x+110y+52y^2-13x^2-48xy+5x^3-5x^2y-x^4+8y^3-2x^2y^2+2x^3y-12xy^2.
$$
If we do not exploit the identical denominators, we obtain a superset of the support.
\end{example}

\begin{example}\label{Exam:samedenom_5vert}
$$
x=\frac{t^6+2t^2}{t^7+1} , y=\frac{t^4-t^3}{t^7+1} ,
$$
hence $B_0=\{2,6\},B_1=\{3,4\},B_2=\{0,7\}$, so this is case~(2A) with $B_2=[0,u]$.
In theorem~\ref{Tcase2a}, we set $i=0,j=1$ and obtain the implicit points
$(e_0,e_1)=(7,0),(0,7),(0,3),(3,1),(6,0)$, in the order stated by the theorem.
These are also the implicit vertices and define the optimal polygon
because the implicit equation is\\[-0.5em]

$
\phi= -32y^4-30x^3y^2-x^4y-12x^2y^2-3x^3y-7x^6y-2x^7+20xy^3+280x^2y^5-7^3y^4x-70x^4y^3-\\
22x^3y^3-49x^5y^2-21x^4y^2+11x^5y+216y^5+129y^7-248y^6+70xy^6+185xy^5+24y^3+100xy^4+\\
43x^2y^3+72x^2y^4+3x^6.\\[-0.5em]
$

In figure~\ref{Fig:Nf_exam_intro} is shown the implicit polygon.
Changing the coefficient of $t^2$ to -1, leads to an implicit polygon
with 6 vertices $(1,3),(0,4),(0,6),(2,5),(7,0),(4,1)$,
which is contained in the polygon predicted by theorem~\ref{Tcase2a}.
This shows the importance of the genericity condition on the coefficients of the
parametric polynomials.
\end{example}

\section{Polynomial parameterizations}\label{Spolynomial}

We consider polynomial parameterizations of curves. In this case we define polynomials
$$
f_0=x-P_0(t),~ f_1=y-P_1(t) \in (\C[x,y])[t].
$$
The supports of $f_0, f_1$ are fixed, namely $A_0=\{a_{00},a_{01},\ldots,a_{0n}\}$ and
$A_1=\{a_{10},a_{11},\ldots,a_{1m}\}$, with generic coefficients.
Here, $a_{0i}$ and $a_{1j}$ are sorted in ascending order.
Points $a_{00}, a_{10}$ are always equal to zero.
The new point set
$$
C=\kappa(A_0,A_1)=\{(a_{00},0),\ldots,(a_{0n},0),(a_{10},1),\ldots,(a_{1m},1)\},\in\N^2
$$
is introduced by the Cayley embedding.
For convenience, we shall omit the second coordinate.
Every triangulation of this set is regular,
and corresponds to a mixed cell configuration of $A_0 + A_1$.

The resultant $\RR(f_0,f_1)$ is a polynomial in $x,y,c_{ij}$,
where $c_{ij}$ are the coefficients of the polynomials $P_i$.
We consider the specialization of coefficients $c_{ij}$ in the resultant.
Generically, this specialization yields the implicit equation.
Now, $N(\phi)\subset\Z^2$ with vertices obtained from those
extreme monomials of $\RR(f_0,f_1)$ which contain coefficients of $a_{00}$ and $a_{10}$.
Since every triangle of a triangulation $T$ of $C$ corresponds to a mixed cell
of a mixed subdivision of $A_0+A_1$, we can rewrite relation (\ref{Eq:Sturmf_extreme}) as:
\begin{equation} \label{EQxyPolynomial}
\pm\prod_{i=0}^1 \prod_R c_{i,p}^{\mathrm{Vol}(R)},
\end{equation}
where $R$ is an $i$-mixed cell with vertex $p \in A_i$ and $c_{i,p}$
is the coefficient of the monomial with exponent $p\in A_i$.
After specialization of the coefficients of $f_0,f_1$, the terms
of~(\ref{EQxyPolynomial}) associated with mixed cells having a vertex $p$ other than
$a_{00}, a_{10}$ contribute only a constant to the corresponding term.
This implies that the only mixed cells that we need to consider are the ones with vertex
$a_{00}$ or $a_{10}$ (or both).
For any triangulation $T$, these mixed cells correspond to triangles  with vertices
$a_{00},a_{1l},a_{1r}$ where $l,r\in\{0,\dots,m\}$, or $a_{10},a_{0l},a_{0r}$, where $l,r\in\{0,\dots,n\}$.

The first statement below can be obtained from the degree bounds; we establish it by our
methods for completeness.

\begin{thm}\label{T:Nfpolynomial}
If $P_0$ or $P_1$ (or both) contain a constant term, then
the implicit polygon is the triangle with vertices $(0,0), (a_{1m},0), (0,a_{0n})$.
Otherwise, $P_0,P_1$ contain no constant terms, and the implicit polygon
is the quadrilateral with vertices $(a_{11},0), (a_{1m},0)$, $(0,a_{0n}), (0,a_{01}) $.

\begin{proof}
Let us consider the first statement.
To obtain vertices $(a_{1m},0)$ and $(0,0)$ consider the triangulation $T$ of $C$ obtained by drawing edge $(a_{00},a_{1m})$
(see figure~\ref{Fig:Nf_polynomial}).
The only 0-mixed cell with vertex $a_{00}$ corresponding to $T$ is $R=a_{00}+(a_{10},a_{1m})$ with volume equal to $a_{1m}$;
 there are no 1-mixed cells with vertex $a_{10}$.
The extreme monomial associated with such a triangulation is of the form $(x-c_{00})^{a_{1m}} c_{1m}^{a_{0n}}$ which,
 after specializing $c_{00}, c_{1m}$ and expanding, gives monomials in $x$ with exponents $a_{1m}, a_{1m-1},\ldots,0$.

For vertex $(0,a_{0n})$ consider triangulation $T'$  obtained by drawing edge  $(a_{10},a_{0n})$.
The only 1-mixed cell with vertex $a_{10}$  is $R=a_{10}+(a_{00},a_{0n})$ with volume equal to $a_{0n}$;
there are no 0-mixed cells with vertex $a_{00}$.
The extreme monomial is  $(y-c_{10})^{a_{0n}} c_{0n}^{a_{1m}}$ which, after specializing $c_{10}, c_{0n}$ and expanding,
 gives monomials in $y$ with exponents $a_{0n}, a_{0n-1},\ldots,0$.

To complete the proof it suffices to observe that every triangulation of $C$ having edges of the form $(a_{00},a_{1j}), ~0<j<m$ and $(a_{1j},a_{0i})$, $i>0$ leads to an extreme monomial which specializes to a polynomial in $x$ of degree $a_{1j}$. Therefore we obtain monomials  with exponents $(a_{1j},0),\ldots,(0,0)$ which lie in the interior of the triangle.
Similarly, every triangulation of $C$ having edges of the form $(a_{10},a_{0i}), ~0<i<n$ and $(a_{0i},_{1j})$, $j>0$ leads to an extreme monomial which specializes to a polynomial in $y$ of degree $a_{0i}$. Therefore we obtain monomials with exponents $(0,a_{0i}),\ldots,(0,0)$ which lie in the interior of the triangle.

The extreme monomials associated with
$a_{00}$ and $a_{10}$ are specialized to monomials of the implicit equation in $x$ or $y$ respectively,
thus not producing any constant terms.
The proof of the previous lemma implies that, when $y$'s exponent is $0$, the smallest exponent of $x$
is $a_{11}$, which is obtained by a triangulation containing edges $(a_{00},a_{11})$ and
$(a_{11},a_{0i}), ~ i>0$.
Similarly, the smallest exponent of $y$ is $a_{01}$. 
\end{proof}
\end{thm}

\begin{figure}[ht]
\psfrag{a0}{\tiny $a_{00}$}
\psfrag{a1}{\tiny $a_{01}$}
\psfrag{an}{\tiny $a_{0n}$}
\psfrag{b0}{\tiny $a_{10}$}
\psfrag{b1}{\tiny $a_{11}$}
\psfrag{bm}{\tiny $a_{1m}$}
\psfrag{YM}{\tiny $(0,a_{0n})$}
\psfrag{Ym}{\tiny $(0,a_{01})$}
\psfrag{XM}{\tiny $(a_{1m},0)$}
\psfrag{Xm}{\tiny $(a_{11},0)$}
\centering \includegraphics[width=0.6\textwidth]{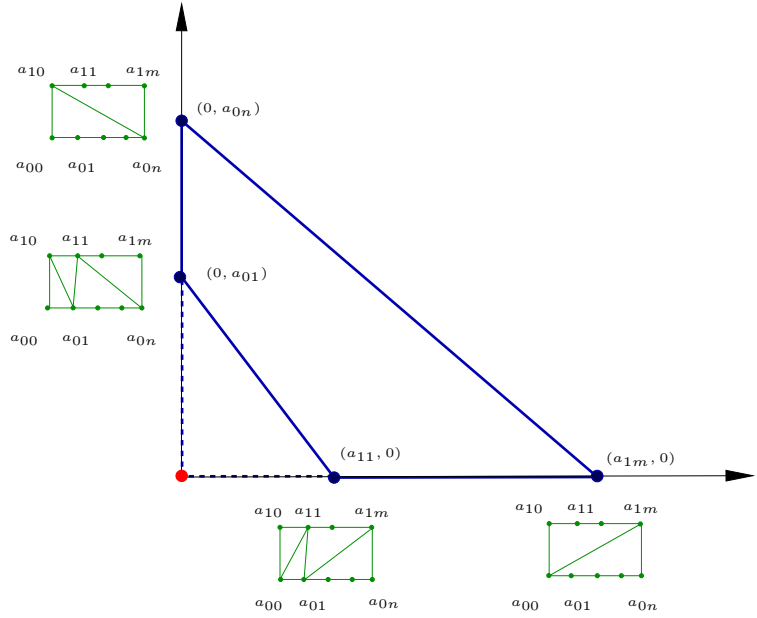}
\caption{The implicit polygon of a polynomially parameterized curve.}
\label{Fig:Nf_polynomial}
\end{figure}

Now we use~\cite[Prop.15]{GKZ90} to arrive at the following;
recall that the implicit equation is defined up to a sign.
The coefficient of $x^{a_{1m}}$ is $c(-1)^{(1+a_{0n})a_{1m}} c_{1m}^{a_{0n}}$
and that of $y^{a_{0n}}$ is $c(-1)^{a_{0n}(1+a_{1m})} c_{0n}^{a_{1m}}$,
where $c\in\{-1,1\}$.

\begin{cor}\label{Ccoef}
There exists $c\in \{-1,1\}$ s.t. the coefficient of $x^{a_{1m}}$ is $c(-c_{1m})^{a_{0n}}$ and that of $y^{a_{0n}}$ is
$c(-c_{0n})^{a_{1m}}$.
\end{cor}

\begin{example}\label{Ex:poly_degenerate}
Parameterization $x=y=t$ yields implicit equation $\phi=x-y$.
Our method yields vertices $(1,0)$ and $(0,1)$ which are optimal.
\end{example}

\begin{example}
Parameterization  $x=2t^3-t+1, y=t^4-2t^2+3$ yields implicit equation
 $\phi= 608 - 136x + 569y +168y^2 -72x^2 -32xy -4x^3 -16x^2y -x^4 +16y^3$.
Our method yields the vertices $(0,0), (4,0), (0,3)$ which are optimal.
The degree bounds describe a larger quadrilateral with vertices $(0,0),(4,0),(1,3),(0,3)$.
Corollary~\ref{Ccoef} predicts, for $x^4$, coefficient $(-1)^{16} =1$,
and for $y^3$, coefficient $(-1)^{15} 2^4= -16$, up to a fixed sign
which equals $-1$ in $\phi(x,y)$.
\end{example}

\begin{example}
For the Fr\"oberg-Dickenstein example \cite[Exam.3.3]{EmiKot05},
$$
x=t^{48}-t^{56}-t^{60}-t^{62}-t^{63}, y=t^{32},
$$
our method yields vertices $(32,0), (0,48), (0,63)$, which
define the optimal polygon.
Here the degree bounds describe the larger quadrilateral with vertices
$(0,0),(32,0),(32,31), (0,63)$.
\end{example}

\begin{example}
Parameterization $x=t+t^2, y=2t-t^2$ yields implicit equation $\phi= 6x-3y+x^2+2xy+y^2$.
The previous lemma yields vertices $(1,0),(2,0),(0,2),(0,1)$, which define
the actual implicit polygon.
Here the degree bounds imply a larger triangle, with vertices $(0,0),(2,0),(0,2)$.
Corollary~\ref{Ccoef} predicts, for $x^{2}$ and $y^2$, coefficients
$(-1)^{6} (-1)^{2}=1$ and $(-1)^{6} (1)^{2}=1$ respectively.
\end{example}

\section{Rational parameterizations  with different denominators}\label{S:Rat_dif_denom}

Now we turn to the case of rationally parameterized curves, with different denominators.
We have
$$
f_0(t)=xQ_0(t)-P_0(t),\ f_1(t)=yQ_1(t)-P_1(t)\ \in (\C[x,y])[t], \, \gcd(P_i,Q_i)=1,
$$
and let $c_{ij} ~(0\leq j \leq m_i)$, ~$q_{ij} ~(0\leq j \leq k)$
denote the coefficients of polynomials $P_i(\t)$ and $Q_i(t)$, respectively.
The supports of $f_0, f_1$ are of the form $A_0=\{a_{00},a_{01},\ldots,a_{0n}\}$ and
$A_1=\{a_{10},a_{11},\ldots,a_{1m}\}$ where the $a_{0i}$ and $a_{1j}$ are sorted in ascending order;
 $a_{00} = a_{10}=0$ because  $\gcd(P_i,Q_i)=1$.
Points in $A_0, A_1$ are embedded by $\kappa$ in $\R^2$.
The embedded points are denoted by $(a_{0i},0), (a_{1i},1)$; by abusing notation, we will omit the extra coordinate.

Recall that each $p\in A_0$ corresponds to a monomial of $f_0$.
The corresponding coefficient either lies
 in $\C$, or is a monomial
$q_i x$, or a binomial $q_i x+c_{0i}$, where  $q_i,c_{0i}\in \C$.
An analogous description holds for the second polynomial.

\begin{defn}\label{D:selection}
Let $V,W$ be non-empty subsets of $\Z$.
A {\em selection} is a pair of sets $S, T$
such that $S \subseteq V$ and $T \subseteq W$.
We say that the elements of the sets $S$ and $T$ are selected, and that
the elements of $V\setminus S$ and $W\setminus T$ are non-selected.
\end{defn}

With respect to the sets~$A_0$ and $A_1$, we now define two types of
selections:
\begin{itemize}
\item
\textit{Selection1}:
the exponents in $A_0$ and $A_1$ corresponding to coefficients
which are non-constant polynomials (i.e., they are either linear monomials
or linear binomials) in $\C[x]$ and $\C[y]$, respectively, are selected.
Let $S_i \subseteq A_i$ ($i=0,1$) be the sets of the selection.
The selected exponents in $S_i$ are those in the support of the denominator $Q_i(t)$;
moreover, $|S_0| \ge 1$ and $|S_1| \ge 1$, i.e., at least one exponent from
both $A_0$ and $A_1$ is selected since $Q_i(t)\neq 0$.
\item
\textit{Selection2}:
the exponents in $A_0$ and $A_1$ corresponding to coefficients
which are monomials in $\C[x]$ and $\C[y]$, respectively, are selected.
Let $S'_i \subseteq A_i$ ($i=0,1$) be the sets of the selection.
In this case, $S'_i={\rm supp}(Q_i)\setminus {\rm supp}(P_i)$; there is
at least one non-selected exponent in $A_0$ and in $A_1$ coming from
the numerator $P_i(t)$.
\end{itemize}
In order to denote that $a_{0i} \in A_0$ or $a_{1i} \in A_1$ is selected
(non-selected, resp.), we write $a_{0i}^+$ or $a_{1i}^+$ ~($a_{0i}^-$ or $a_{1i}^-$, resp.).
For example, the case of polynomial parameterizations yields
$A_0=\{a_{00}^+,a_{01}^-,\dots,a_{0n}^-\}$, $A_1=\{a_{10}^+,a_{11}^-,\dots,a_{1m}^-\}$, under
Selection1.

We shall consider only $i$-mixed cells associated with a selected vertex in $A_i$.
For any triangulation $T$, these mixed cells correspond either to triangles with vertices
$\{a_{0i}^+,a_{1{\ell}},a_{1r}\}$, where ${\ell},r\in\{0,\dots,m\}$, or to
$\{a_{0{\ell}},a_{0r},a_{1j}^+\}$, where ${\ell},r\in\{0,\dots,n\}$.
Given a selection and a triangulation, we set
\begin{equation} \label{EQxy}
e_0=\sum_{i,\ell,r} {\rm Vol}(a_{0i}^+,a_{1{\ell}},a_{1r}),
\qquad\qquad e_1=\sum_{\ell,r,j} {\rm Vol} (a_{0{\ell}},a_{0r},a_{1j}^+),
\end{equation}
where $i,j$ range over all selected points in $A_0$ and $A_1$, respectively,
and we sum up the normalized volumes of mixed triangles.

In the following, we use the upper (lower, resp.) hull of a
convex polygon in $\R^2$ wrt some direction $v\in\R^2$.
Let us consider the unbounded convex polygons defined
by the computed upper and lower hulls. The intersection of these two unbounded polygons is
the implicit Newton polygon.

The resultant $\RR(f_0,f_1)$ is a polynomial in $x,y,c_{ij},q_k$.
We consider the specialization of coefficients $c_{ij}, q_k$ in order to
study $\phi$; this specialization yields the implicit equation.
The relevant terms are products of one polynomial in $x$ and one in $y$.
The former is the product of powers of
terms of the form $q_ix$ or $q_ix + c_{0j}$ ; the $y$-polynomial is obtained analogously.

\begin{lem}
Consider all points $(e_0,e_1)$ defined by expressions~(\ref{EQxy}).
The polygon defined by the upper hull of points $(e_0,e_1)$ under
Selection1 and the lower hull of points $(e_0,e_1)$ under Selection2
equals the implicit polygon $N(\phi)$.
\end{lem}

\begin{proof}
Consider the extreme terms of the resultant,
given by thm~\ref{P:Sturmf_extreme} and expression (\ref{EQxyPolynomial}).
After the specialization of the coefficients, those
associated with $i$-mixed cells having a non-selected vertex $p\in A_i$
contribute only a coefficient in $\C$ to the corresponding term of $\phi$.
This is why they are not taken into account in~(\ref{EQxy}).

Now consider Selection1.
By maximizing $e_0$ or $e_1$, as defined in~(\ref{EQxy}), it is clear that
we shall obtain the maximum possible exponents in the terms which are
polynomials in $x$ and $y$ respectively, hence the largest degrees in $x,y$ in $\phi$.
Under certain genericity assumptions, we shall obtain all vertices in the implicit polygon,
which appear in its upper hull with respect to vector $(1,1)$.
If genericity fails, the implicit polygon will contain vertices with smaller coordinates.

Selection2 minimizes the powers of coefficients corresponding
to monomials in the implicit variables.
All other coefficients are in $\C$ or are binomials in $x$ (or $y$), so
they contain a constant term, hence their product will contain a
constant, assuming generic coefficients in the parametric equations.
Therefore these are vertices on the lower hull with respect to $(1,1)$.
If genericity fails, then fewer terms appear in $\phi$ and the implicit polygon
is interior to the lower hull computed.
\end{proof}

\subsection{The implicit vertices}

For a set $P$ and any $p \in P$, we define functions $\mathcal{X}(p^+)$  and $\mathcal{X}(P^-)$  where
$\mathcal{X}(p^+)=1$ if $p$ is selected and $\mathcal{X}(p^+)=0$ otherwise,
and $\mathcal{X}(P^-)=1$ if there exists some non-selected point $p^- \in P$ and $\mathcal{X}(P^-)=0$  otherwise.
Function $\mathcal{X}(P^-)$ satisfies
${\mathcal X}(P^-) = 1 -\prod_i {\mathcal X}(p_i^+)$.
Recall that $a_{00}=a_{10}=0$; nevertheless, we still use $a_{00},a_{10}$ for generality.

The following two lemmas describe the upper hull defined by expressions~(\ref{EQxy}).

\begin{lem}\label{L:rational_x-max,y-max}
The maximum exponent  of $x$ in the implicit equation is
$e_0^{max}=a_{1m} - a_{10}.$
When this is attained, the maximum exponent of $y$ is
$$e_1^{max}|_{e_0^{max}}= (a_{0R}^+ - a_{0L}^+) + \mathcal{X}(a_{10}^+)\cdot(a_{0L}^+ - a_{00}) + \mathcal{X}(a_{1m}^+)\cdot(a_{0n} - a_{0R}^+),$$
where $a_{0R}^+, a_{0L}^+$ are the rightmost and leftmost selected points
(not necessarily distinct) in $A_0$, with respect to Selection1.
A symmetric result holds for $e_0^{max}|_{e_1^{max}}$.
\begin{proof}
There always is at least one selected point $a_{0j}^+ \in A_0$ and $a_{1i}^+ \in A_1$.
This implies that the maximum exponent of $x$ is equal to $a_{1m} - a_{10}$ and
is attained by the triangulation with edges $(a_{0j}^+,a_{10}), (a_{0j}^+,a_{1m})$.
Then, the maximum exponent of $y$ is attained
from any triangulation such that a maximum part
of segment $(a_{00},a_{0n})$ is visible from some selected points in $A_1$.
Such a triangulation must contain edges $(a_{0L}^+,a_{10})$ and $(a_{0R}^+,a_{1m})$
(see Figure~\ref{Fig:max_xy_rat}).

Assume that  $\mathcal{X}(a_{10}^+)=1$.
If all other  selected points in $A_1$ (if any) lie inside $(a_{10},a_{1m})$, then $\mathcal{X}(a_{1m}^+)=0$;
the maximum exponent of $y$ is $a_{0R}^+-a_{00}$; it is obtained by drawing edge $(a_{10},a_{0R}^+)$.
If  $a_{1m}$ is also selected, then $\mathcal{X}(a_{1m}^+)=1$
and segment $(a_{0R}^+,a_{0n})$ is also visible from selected points in $A_1$ (namely $a_{1m}$)
hence the maximum exponent of $y$ is  $(a_{0R}^+ - a_{00}) + (a_{0n} - a_{0R}^+) = a_{0n} - a_{00}$.

Assume that $\mathcal{X}(a_{10}^+)=0$.
If all  selected points in $A_1$ lie inside $(a_{10},a_{1m})$, then $\mathcal{X}(a_{1m}^+)=0$
and the maximum exponent of $y$ is $a_{0R}^+-a_{0L}^+$.
It is obtained by drawing edges  $(a_{1i}^+,a_{0L}^+), (a_{1i}^+,a_{0R}^+)$, from some selected point $a_{1i}^+ \in A_1$.
If  $\mathcal{X}(a_{1m}^+)=1$,  segment $(a_{0R}^+,a_{0n})$ is also visible
from selected points in $A_1$  (namely $a_{1m}$)
hence the maximum exponent of $y$ is  $(a_{0R}^+ -a_{0L}^+) + (a_{0n} - a_{0R}^+) = a_{0n} - a_{0L}^+$.
\end{proof}
\end{lem}

%
\begin{figure}[ht]
\psfrag{a0}{\footnotesize $a_{00}$} \psfrag{ai}{\footnotesize $a_{0i}^+$} \psfrag{an}{\footnotesize $a_{0n}$}
\psfrag{an+}{\footnotesize $a_{0n}^+$} \psfrag{an-}{\footnotesize $a_{0n}^-$}
\psfrag{aR}{\footnotesize $a_{0R}^+$} \psfrag{aL}{\footnotesize $a_{0L}^+$}
\psfrag{bi}{\footnotesize $a_{1i}^+$} \psfrag{b0}{\footnotesize $a_{10}$} \psfrag{bm}{\footnotesize $a_{1m}$}
\psfrag{bm-}{\footnotesize $a_{1m}^-$} \psfrag{bm+}{\footnotesize $a_{1m}^+$}
\psfrag{bR}{\footnotesize $a_{1R}^+$} \psfrag{bL}{\footnotesize $a_{1L}^+$}
\centering \includegraphics[width=0.8\textwidth]{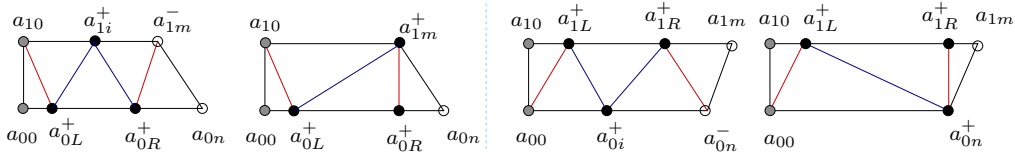}
\caption{The triangulations of $C$ giving vertices $e_1^{max} |_{e_0^{max}}$  (left subfigure)
and $e_0^{max} |_{e_1^{max}}$ (right subfigure).}
\label{Fig:max_xy_rat}
\end{figure}

\begin{lem}\label{L:rational_x-max,y-min}
Suppose that the maximum exponent $e_0^{max}$ of $x$ equal to $a_{1m} - a_{10}$ is attained;
then the minimum exponent of $y$ is
$$
e_1^{min}|_{e_0^{max}} = \mathcal{X}(a_{10}^+)\cdot(a_{0L}^+ - a_{00}) + \mathcal{X}(a_{1m}^+)\cdot(a_{0n} - a_{0R}^+) +
 (1-\mathcal{X}(A_1^-))\cdot (a_{0R}^+ - a_{0L}^+)
$$
where $a_{0R}^+, a_{0L}^+$ are the rightmost and leftmost selected points in $A_0$ with respect to Selection1.
A symmetric result holds for $e_0^{min}|_{e_1^{max}}$.
\begin{proof}
To attain the maximum exponent of $x$ equal to $a_{1m} - a_{10}$,
we have to draw edges $(a_{0i}^+,a_{10}), (a_{0i}^+,a_{1m})$, where $a_{0i}^+$ is some selected point in $A_0$.
An analogous  reasoning as before asks for the minimization of the segment of $(a_{00},a_{0n})$
which is visible from selected points in $A_1$.
We can minimize this segment by drawing edges from non-selected points (if any)
 $a_{1i}^- \in A_1$ to the leftmost and rightmost selected points in $A_0$.
The rest of the proof is similar to that of lemma \ref{L:rational_x-max,y-max}.
\end{proof}
\end{lem}

\begin{figure}[ht]
\psfrag{a0}{\footnotesize $a_{00}$}
\psfrag{ai}{\footnotesize $a_{0i}^-$}
\psfrag{an}{\footnotesize $a_{0n}$}
\psfrag{an+}{\footnotesize $a_{0n}^+$}
\psfrag{an-}{\footnotesize $a_{0n}^-$}
\psfrag{aR}{\footnotesize $a_{0R}^+$}
\psfrag{aL}{\footnotesize $a_{0L}^+$}
\psfrag{bi}{\footnotesize $a_{1i}^-$}
\psfrag{b0}{\footnotesize $a_{10}$}
\psfrag{bm}{\footnotesize $a_{1m}$}
\psfrag{bm-}{\footnotesize $a_{1m}^-$}
\psfrag{bm+}{\footnotesize $a_{1m}^+$}
\psfrag{bR}{\footnotesize $a_{1R}^+$}
\psfrag{bL}{\footnotesize $a_{1L}^+$}
\centering \includegraphics[width=0.8\textwidth]{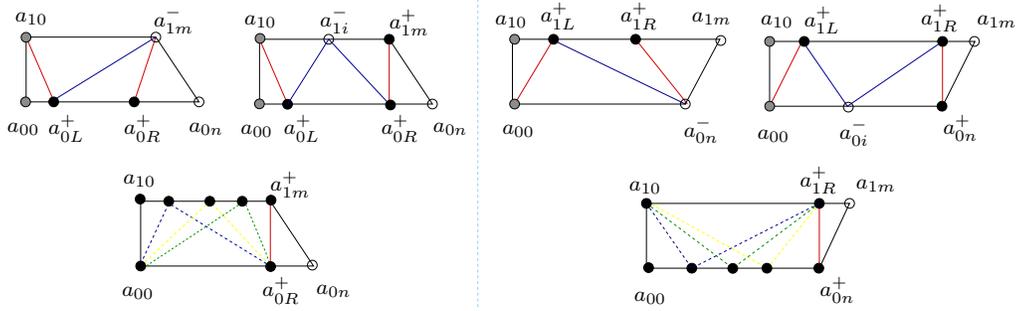}
\caption{The triangulations of $C$ that give the points $e_1^{min} |_{e_0^{max}}$  (left subfigure)
 and $e_0^{min} |_{e_1^{max}}$ (right subfigure).}
\label{Fig:min_xy_rat}
\end{figure}

Now we describe the lower hull defined by expressions~(\ref{EQxy}).

\begin{lem}\label{L:second-criterion-special}
Consider Selection2 and suppose that no point in $A_0$ is selected.
Then,
\begin{itemize}
\item[(i)]
if no point in $A_1$ is selected, the lower hull contains only vertex~$(0,0)$;
\item[(ii)]
if there exists at least one selected and at least one non-selected point
in $A_1$, the lower hull contains only vertices $(0,0)$ and $(0, a_{0n} - a_{00})$.
\end{itemize}
\end{lem}
\noindent
It is not difficult to see that the lemma holds.
A similar result holds if no point in $A_1$ is selected.

In the following, we assume that there exists at least one selected point
in each of the sets~$A_0$ and $A_1$.
Moreover, since we consider Selection2, there exists at least one non-selected
point in each of $A_0$ and $A_1$ as well.

\begin{lem}\label{L:rational_y-min,x-max}
When the exponent of $x$ attains its minimum value $e_0^{min} = 0$,
the maximum exponent of $y$ is
$$
e_1^{max}|_{e_0^{min}}
\  = \  a_{0R}^- - a_{0L}^-
+ \mathcal{X}(a_{10}^+) \cdot (a_{0L}^- - a_{00})
+ \mathcal{X}(a_{1m}^+) \cdot (a_{0n} - a_{0R}^-),
$$
where $a_{0L}^-, a_{0R}^-$ are the leftmost and rightmost non-selected points
in $A_0$ under Selection2.
A symmetric result holds for $e_0^{max}|_{e_1^{min}}$.
\end{lem}

\noindent
The proof of lemma~\ref{L:rational_y-min,x-max} is similar to the proof of
lemma~\ref{L:rational_x-max,y-max}; the only difference is that we focus
on the non-selected points instead of the selected points in $A_0$.

\begin{lem}\label{L:rational_x-min,y-min}
When the exponent of $x$ attains its minimum value $e_0^{min} = 0$,
the minimum exponent of $y$ is
$$
e_1^{min}|_{e_0^{min}} = \mathcal{X}(a_{10}^+) \cdot (a_{0L}^- - a_{00})
+ \mathcal{X}(a_{1m}^+) \cdot (a_{0n} - a_{0R}^-),
$$
where $a_{0L}^-, a_{0R}^-$ are the leftmost and rightmost non-selected points in $A_0$
under Selection2.
A symmetric result holds for $e_0^{min}|_{e_1^{min}}$.
\end{lem}
\noindent
The proof is similar to that of lemma~\ref{L:rational_x-max,y-min} except
that we concentrate on the non-selected points in $A_0$ instead of
the selected ones; note also that there is a non-selected point in $A_1$
and thus $1 - {\mathcal X}(A_1^-) = 0$.

The above lemmas lead to the following results for the four corners of
$N(\phi)$:

\begin{thm}\label{T:rational_upper-left-corner}
Suppose that $e_1^{max}|_{e_0^{min}} \ne a_{0n}-a_{00}$ and let
$
\delta= (a_{0n} - a_{0R}^-) \cdot (a_{1L}^+ - a_{10}) - (a_{0L}^- - a_{00}) \cdot (a_{1m} - a_{1R}^+).
$
Then, under Selection2,
\begin{itemize}
\item
$a_{00}$ is selected and $a_{10}$ is not, or $a_{0n}$ is selected and $a_{1m}$ is not,
which also implies that $e_0^{min}|_{e_1^{max}} \ne 0$;
\item
the upper left corner of $N(\phi)$
consists of a single edge connecting the points $(0, e_1^{max}|_{e_0^{min}})$
and $(e_0^{min}|_{e_1^{max}},$ $a_{0n}-a_{00})$  
unless $a_{00}, a_{0n}$ are selected, $a_{10}, a_{1m}$ are non-selected, and
$\delta \ne 0$, in which case the corner consists of two edges connecting
point $(0, e_1^{max}|_{e_0^{min}})$ to a point~$p$,
and point~$p$ to $(e_0^{min}|_{e_1^{max}}, a_{0n}-a_{00})$, where
$$
p = (a_{1L}^+ - a_{10}, a_{0R}^- - a_{00}) \mbox{ if } \delta < 0
\mbox{\quad and \quad} p = (a_{1m} - a_{1R}^+, a_{0n} - a_{0L}^-) \mbox{ if } \delta > 0.
$$
\end{itemize}
Point $p$ lies on the polygon edge iff $\delta=0$.
A symmetric result holds for the lower right corner.
\begin{proof}
Lemma~\ref{L:rational_y-min,x-max} implies that
$e_1^{max}|_{e_0^{min}} = a_{0n}-a_{00}$ in all cases
except if $a_{00}$ is selected and $a_{10}$ is not, or if
$a_{0n}$ is selected and $a_{1m}$ is not
(note that if $a_{00}$ is selected and $a_{10}$ is not, then
$a_{0L}^- \ne a_{00}$ and $\mathcal{X}(a_{10}^+) \cdot (a_{0L}^- - a_{00}) = 0$,
and similarly if $a_{0n}$ is selected and $a_{1m}$ is not, then
$a_{0R}^- \ne a_{0n}$ and $\mathcal{X}(a_{1m}^+) \cdot (a_{0n} - a_{0R}^-) = 0$).
In each of these cases,
lemma~\ref{L:rational_x-max,y-min} implies that
$e_0^{min}|_{e_1^{max}} \ne 0$.

Let us consider the case in which $a_{00}$ is selected, $a_{10}$ is not,
and $a_{0n}$ is not selected or $a_{1m}$ is selected or both.
Then, $e_1^{max}|_{e_0^{min}} = a_{0n}-a_{0L}^-$
and $e_0^{min}|_{e_1^{max}} = a_{1L}^+ - a_{10}$.
Suppose, for contradiction, that there exists a triangulation~$T$ corresponding
to a point $p_T = (x_T,y_T)$ with $x_T < a_{1L}^+ - a_{10}$ and $y_T > a_{0n} - a_{0L}^-$.
Consider the edges $a_{0i} a_{1j}$ of $T$; as these edges do not cross, they
can be ordered from left to right. The leftmost edge is $a_{00} a_{10}$ with
$a_{00}$ selected and $a_{10}$ not selected. Let $a_{0i} a_{1j}$ be the leftmost edge
such that either $a_{0i}$ is not selected or $a_{1j}$ is selected;
exactly one of these two conditions will hold, since $a_{0i} a_{1j}$ is
the leftmost such edge of a triangulation,
If $a_{0i}$ is not selected, then all the points $a_{10}, \ldots, a_{1j}$ are
not selected, and thus no portion of the segment $(a_{00},a_{0i})$ contributes
to the $y$-coordinate~$y_t$ of $p_T$, i.e., $y_T \le a_{0n} - a_{0i} \le a_{0n} - a_{0L}^-$,
a contradiction.
Similarly, if $a_{1j}$ is selected, then all the points $a_{00}, \ldots, a_{0i}$
are selected, and thus the entire segment $(a_{10}, a_{1j})$ contributes to
the $x$-coordinate~$x_t$, i.e., $x_T \ge a_{1j} - a_{10} \ge a_{1L}^+ - a_{10}$,
a contradiction again.
Therefore, the upper left corner in this case consists of the edge
connecting $(0, a_{0n}-a_{0L}^-)$ and $(a_{1L}^+ - a_{10}, a_{0n}-a_{00})$.
The case in which $a_{0n}$ is selected, $a_{1m}$ is not, and $a_{00}$ is not selected or
$a_{10}$ is selected or both is right-to-left symmetric yielding a similar result.

Finally, we consider the case in which $a_{00}$ and $a_{0n}$ are selected
and $a_{10}$ and $a_{1m}$ are not.
Then, $e_1^{max}|_{e_0^{min}} = a_{0R}^- - a_{0L}^-$
and $e_0^{min}|_{e_1^{max}} = a_{1L}^+ - a_{10} + a_{1m} - a_{1R}^+$ leading to points
$q_1 = (0, a_{0R}^- - a_{0L}^-)$ and $q_2 = (a_{1L}^+ - a_{10} + a_{1m} - a_{1R}^+, a_{0n} - a_{00})$.
Let us consider the points $p_1 = (a_{1L}^+ - a_{10}, a_{0R}^- - a_{00})$
and $p_2 = (a_{1m} - a_{1R}^+, a_{0n} - a_{0L}^-)$.
It is not difficult to see that one can obtain triangulations
corresponding to these points; for $p_1$, we add the edges
$a_{00} a_{1L}^+$, $a_{1L}^+ a_{0R}^-$, and $a_{0R}^- a_{1m}$, while for $p_2$
the edges $a_{0n} a_{1R}^+$, $a_{1R}^+ a_{0L}^-$, and $a_{0L}^- a_{10}$.
Moreover, the points $q_1$, $q_2$, $p_1$, and $p_2$ form a parallelogram
which degenerates to a line segment if $\delta = 0$; if $\delta \ne 0$,
then $p_1$ ($p_2$, resp.) is above the line through $q_1$ and $q_2$ if
$\delta < 0$ ($\delta > 0$, resp.).
Assume for the moment that $\delta < 0$. We will show that the edges $q_1 p_1$
is an edge of $N(\phi)$; suppose, for contradiction, that
there exists a triangulation~$T$ corresponding to a point~$p_T = (x_T,y_T)$
which has $x_T < a_{1L}^+ - a_{10}$, $y_T > a_{0R}^- - a_{0L}^-$ and lies above
the line through $q_1$ and $p_1$.
Since $a_{00}$ is selected and $a_{10}$ is not, we can consider the ordered
edges $a_{0i} a_{1j}$ of $T$ (from left to right) and we can show as above that
either the entire segment $(a_{10}, a_{1L}^+)$ contributes to the $x$-coordinate
of $p_T$ or no part of the segment $(a_{00}, a_{0L}^-)$
contributes to its $y$-coordinate; the former is in contradiction
with the fact that $x_T < a_{1L}^+ - a_{10}$, and thus the latter case holds.
Moreover, by considering the edges $a_{0i} a_{1j}$ of $T$ from right to left,
we can show that
either the entire segment $(a_{1R}^+, a_{1m})$ contributes to the $x$-coordinate
of $p_T$ or no part of the segment $(a_{0R}^-, a_{0n})$ contributes to its
$y$-coordinate; the latter case, in conjunction with the latter case of
the previous observation, is in contradiction with $y_T > a_{0R}^- - a_{0L}^-$,
and hence the former case holds. Thus, $x_T \ge a_{1m} - a_{1R}^+$ and
$y_T \le a_{0n} - a_{0L}^-$. For $p_T$ to be above the line through $q_1$ and $p_1$,
it should hold that
$\frac{y_T - (a_{0R}^- - a_{0L}^-)}{x_T} > \frac{a_{0L}^- - a_{00}}{a_{1L}^+ - a_{10}}$;
this is not possible because $\delta < 0 \Longrightarrow
\frac{a_{0n} - a_{0R}^-}{a_{1m} - a_{1r}^+} < \frac{a_{0L}^- - a_{00}}{a_{1L}^+ - a_{10}}$ and
$\frac{y_T - (a_{0R}^- - a_{0L}^-)}{x_T} \le \frac{a_{0n} - a_{0R}^-}{a_{1m} - a_{1r}^+}$.
Therefore, the segment $q_1 p_1$ is an edge of $N(\phi)$.
For $\delta < 0$, we can show in a similar fashion that the segment
$q_2 p_1$ is also an edge of $N(\phi)$.
The cases for $\delta > 0$ are symmetric involving point~$p_2$.
\end{proof}
\end{thm}

In a similar fashion, we can show the following theorems:

\begin{thm}\label{T:rational_lower-left-corner}
Suppose that $e_1^{min}|_{e_0^{min}} \ne 0$ and let
$
\delta= (a_{0n} - a_{0R}^-) \cdot (a_{1L}^- - a_{10}) - (a_{0L}^- - a_{00}) \cdot (a_{1m} - a_{1R}^-).
$
Then, under Selection2,
\begin{itemize}
\item
$a_{00}, a_{10}$ are selected, or $a_{0n}, a_{1m}$ are selected,
which also implies that $e_0^{min}|_{e_1^{min}} \ne 0$;
\item
the lower left corner of $N(\phi)$
consists of a single edge connecting the points $(0, e_1^{min}|_{e_0^{min}})$
and $(e_0^{min}|_{e_1^{min}}, 0)$
unless all four points $a_{00}, a_{10}, a_{0n}, a_{1m}$ are selected and $\delta \ne 0$
in which case the corner consists of two edges
connecting $(0, e_1^{min}|_{e_0^{min}})$ to point~$p$, and $p$ to
$(e_0^{min}|_{e_1^{min}}, 0)$, where
$$
p = (a_{1L}^- - a_{10}, a_{0n} - a_{0R}^-) \mbox{ if } \delta < 0
\mbox{\quad and \quad} p = (a_{1m} - a_{1R}^-, a_{0L}^- - a_{00}) \mbox{ if } \delta > 0.
$$
\end{itemize}
\end{thm}

\begin{thm}\label{T:rational_upper-right-corner}
Suppose that $e_1^{max}|_{e_0^{max}} \ne a_{0n} - a_{00}$ and let
$
\delta= (a_{0n} - a_{0R}^+) \cdot (a_{1L}^+ - a_{10}) - (a_{0L}^+ - a_{00}) \cdot (a_{1m} - a_{1R}^+).
$
Then, under Selection1,
\begin{itemize}
\item
none of $a_{00}, a_{10}$ is selected or none of $a_{0n}, a_{1m}$ is selected,
which also implies that $e_0^{max}|_{e_1^{max}} \ne a_{1m} - a_{10}$;
\item
the upper right corner of $N(\phi)$ consists of a single edge connecting
$(a_{1m}- a_{10}, e_1^{max}|_{e_0^{max}})$ and $(e_0^{max}|_{e_1^{max}}, a_{0n} - a_{00})$,
unless none of the $a_{00}, a_{10}, a_{0n}, a_{1m}$ is selected
and $\delta \ne 0$, in which case the corner consists of~2 edges
connecting $(a_{1m} - a_{10}, e_1^{max}|_{e_0^{max}})$ to~$p$, and $p$ to
$(e_0^{max}|_{e_1^{max}}, a_{0n} - a_{00})$, where
$$
p = (a_{1m} - a_{1L}^+, a_{0R}^+ - a_{00}) \mbox{ if } \delta < 0
\mbox{\quad and \quad} p = (a_{1R}^+ - a_{10}, a_{0n} - a_{0L}^+) \mbox{ if } \delta > 0.
$$
\end{itemize}
\end{thm}


\begin{example}\label{Exam:laurent}
$$
x=\frac{ a+t^2 }{ct}, \; y=\frac{ b }{ dt }, \quad a,b,c,d\ne 0 .
$$
With generic coefficients, the denominators are different.
The input supports are $A_0=\{0,1^+,2\}, A_1=\{0,1^+\}$, where
we have indicated the selected points.
In this example, both selection criteria lead to the same (singleton) selected subsets.
The polygon obtained by our method has vertices $\{(0,0),(1,1),(0,2)\}$,
which is optimal since
$$
\phi=a d^2 y^2-b c d xy+b^2 .
$$
\end{example}

\begin{example}\label{EXdiff_denom}
Parameterization
$$
x=\frac{ t^7+t^4+t^3+t^2}{ t^3+1 }, ~ y= \frac{ t^5+t^4+t}{ t^5+t^2+1 }
$$
yields implicit polygon with vertices
$(0,2), (0,7), (1,0), (5,0), (5,7)$
which are the vertices computed by our method.
The supports of $f_0, f_1$ are
$A_0=\{0^+,2^-,3^+,4^-, 7^-\},~ A_1=\{0^+,1^-, 2^+,4^-,5^+\}$
where the notation is under Selection1.
Selection2 gives $A_0=\{0^+,2^-,3^-,4^-, 7^-\},~ A_1=\{0^+,1^-, 2^+,4^-,5^-\}$.
\end{example}

\begin{example}\label{Exam:circle_diff}
For the unit circle, $ x=2t/(t^2+1), ~y=(1-t^2)/(t^2+1)$,
the supports are
$A_0=\{0^+,1^-,2^+\}, ~A_1=\{0^+,2^+\}$, under the first selection and
$A_0=\{0^+,1^-,2^+\}, ~A_1=\{0^-,2^-\}$, under the second selection.
The set $C=\kappa(A_0,A_1)$ has~5 triangulations shown in figure~\ref{Fig:example_circle}
which, after applying prop.~\ref{P:Sturmf_extreme}, give the terms
$y^2-1, ~ x^2y^2-2x^2y+x^2$ and $x^2y^2+2x^2y+x^2$.
This method yields vertices $(2,2), (2,0), (0,2), (0,0)$.
By degree bounds, we end up with vertices $(2,0), (0,2), (0,0)$,
Interestingly, to see the cancellation of term $x^2y^2$ it does not suffice to
consider only terms coming from extremal monomials in the resultant.

\begin{figure}[ht]
\psfrag{a0}{\footnotesize $a_{00}^+$} \psfrag{a1}{\footnotesize $a_{01}^-$}
\psfrag{a2}{\footnotesize $a_{02}^+$}
\psfrag{b0}{\footnotesize $a_{10}^+$} \psfrag{b1}{\footnotesize $a_{11}^+$}
\psfrag{1}{\scriptsize $(y-1)(y+1)$} \psfrag{2}{\scriptsize $(y-1)^2x^2$}
\psfrag{3}{\scriptsize $x^2(y+1)^2$} \psfrag{4}{\scriptsize $(y-1)(y-1)x^2$}
\psfrag{5}{\scriptsize $x^2(y+1)^2$}
\centering \includegraphics[width=0.8\textwidth]{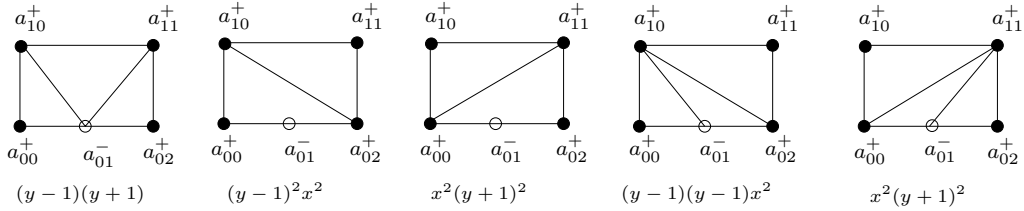}
\caption{The triangulations of $C$ in example~\ref{Exam:circle},
and the corresponding terms (under the first selection).}
\label{Fig:example_circle}
\end{figure}
See example~~\ref{Exam:circle} for a treatment taking into account the
identical denominators.
\end{example}

\begin{example}\label{EX:DAndrea}
Consider the parameterization
$$
x=\frac{ t^3+2t^2+t }{ t^2+3t-2 }, ~ y= \frac{ t^3-t^2 }{ t-2 }.
$$
The supports of $f_0, f_1$ are
$A_0=\{0^-,1^+,2^+,3^-\},~ A_1=\{0^+,1^+,2^-,3^-\}$
where the notation is under Selection1.
Selection2 gives
$A_0=\{0^+,1^-,2^-,3^-\},~ A_1=\{0^+,1^+,2^-,3^-\}$.
Our method yields the implicit support
$\{(0,1), (0,3), (3,0), (1,3), (2,0), (3,2)\}$
which defines the actual implicit polygon.
In figure~\ref{Fig:Nf_exam_intro} is shown the implicit polygon.
\end{example}


\section{Further work}\label{Sfurther}

In conclusion, we have shown that the case of common denominators reduces to a
particular system of~3 bivariate polynomials, where only {\em linear} liftings matter.
An interesting open question is to examine to which systems this observation holds,
since it simplifies the enumeration of mixed subdivisions and, hence, of the
extreme resultant monomials.
In particular, we may ask whether this holds whenever the Newton polytopes are pyramids,
or for systems with separated variables.


It is possible to use our results in deciding which polygons can appear as Newton
polygons of plane curves, and
which parameterization is possible in the generic case.
In particular, theorem~\ref{T:Nfpolynomial} and cor.~\ref{Ccoef}
imply that the Newton polygon of polynomial curves always has one vertex on each axis.
These vertices define the edge that equals
the polygon's upper hull in direction $(1,1)$.
The rest of the edges form the lower hull.
If the implicit polygon is a segment, then the parametric polynomials must be monomials.
Moreover, the implicit polygon cannot contain interior points,
provided the degree of the parameterization is~1 (cf.\ sec.~\ref{Spreliminaries}).
%
Similar results hold for curves parameterized by Laurent polynomials.

By approximating the given polygon by one of the polygons described above,
one might formulate a question of approximate parameterization.

\bigskip

{\small\noindent {\bf Acknowledgement:}
We thank Carlos D'Andrea and Martin Sombra for offering example~\ref{EX:DAndrea}.
The first author acknowledges discussions with Josephine Yu;
part of this work was done while he was visiting the Institute of Mathematics and its
Applications (IMA) at Minneapolis for the Thematic Year on Applications of
Algebraic Geometry and the Workshop on Nonlinear Computational Geometry, in 2007.

The authors are supported through PENED 2003 program, contract nr.\ 70/03/8473.
The program is co-funded by the EU -- European Social Fund (75\% of public funding),
national resources -- General Secretariat of Research and Technology of Greece
(25\% of public funding) as well as the private sector, in the framework of
Measure~8.3 of the Community Support Framework.

}

\bibliographystyle{alpha}
\bibliography{CompNewtPoly}

\end{document}